\documentclass[final]{svmult}

\usepackage{mathptmx}
\usepackage{helvet}
\usepackage{courier}
\usepackage{graphicx}
\usepackage{makeidx}
\usepackage{multicol}
\usepackage{footmisc}

\usepackage{amsmath,amscd}
\usepackage{amssymb}
\usepackage{mathrsfs}
\usepackage{fixme}
\usepackage{pgf}
\usepackage{tikz}
\usetikzlibrary{arrows}
\usetikzlibrary{positioning}
\usetikzlibrary{shapes.geometric}
\usetikzlibrary{calc}

\newcommand{\RJXd}[1]{\mathsf{X}_{\textrm{diag}(#1) }}

\DeclareMathOperator{\RJId}{Id}
\DeclareMathOperator{\RJrl}{rl}
\DeclareMathOperator{\RJleftl}{ll}
\DeclareMathOperator{\RJBF}{BF}
\DeclareMathOperator{\RJdiag}{diag}
\DeclareMathOperator{\RJsgn}{sgn}

\makeindex

\begin{document}

\title*{Is Every Irreducible Shift of Finite Type Flow Equivalent to a Renewal System?}
\author{Rune Johansen}
\institute{Rune Johansen \at Department of Mathematical Sciences, University of Copenhagen, Universitetsparken 5, 2100 K\o benhavn \O, \email{rune@math.ku.dk}} \maketitle

\abstract*{
Is every irreducible shift of finite type flow equivalent to a renewal system? For the first time, this variation of a classic problem formulated by Adler is investigated, and several partial results are obtained in an attempt to find the range of the Bowen--Franks invariant over the set of renewal systems of finite type. In particular, it is shown that the Bowen--Franks group is cyclic for every member of a class of renewal systems known to attain all entropies realised by shifts of finite type, and several classes of renewal systems with non--trivial values of the invariant are constructed.
}

A renewal system is a shift space consisting of the biinfinite sequences that can be obtained as free concatenations of words from some finite generating list. This simple definition hides a surprisingly rich structure that is in many ways independent of the usual topological and dynamical structure of the shift space. 
The present work was motivated by the following problem raised by Adler: Is every irreducible shift of finite type conjugate to a renewal system? Several attempts have been made to answer this question, and the conjugacy of certain special classes of renewal systems is well understood, but there exist only a few results concerning the general problem.
This work is the first to investigate the corresponding question for flow equivalence. The aim has been to find the range of the Bowen--Franks invariant over renewal systems of finite type, and several classes of renewal systems displaying a wide range of values of the invariant are constructed, but it remains unknown whether renewal systems can attain all the values attained by irreducible shifts of finite type. 

Section \ref{RJ_sec_introduction} gives an introduction to shift spaces and renewal systems. Section \ref{RJ_sec_rs_lfc} concerns the left Fischer covers of renewal systems and gives conditions under which the Fischer covers of complicated renewal systems can be constructed from simpler building blocks with known presentations. Section \ref{RJ_sec_rs_entropy} gives a flow classification of a class of renewal systems introduced in \cite{RJ_hong_shin}, while Sec.\ \ref{RJ_sec_rs_range} uses the results of the previous two sections to construct classes of renewal systems with interesting values of the Bowen--Franks invariant.

\emph{Acknowledgements}: Supported by \textsc{VILLUM FONDEN} through the experimental mathematics network at the University of Copenhagen. Supported by the Danish National Research Foundation through the Centre for Symmetry and Deformation (DNRF92).  

\section{Introduction}
\label{RJ_sec_introduction}
\index{shift space}
\index{full shift}
\index{shift map}
\index{shift of finite type}
\index{word}
\index{SFT|see{shift of finite type}} 
Here, a short introduction to the basic definitions and properties of shift
spaces is given to make the present paper self--contained. 
For a thorough treatment of shift spaces see \cite{RJ_lind_marcus}. 
Let $\mathcal{A}$ be a finite set with the discrete topology. The
\emph{full shift} over $\mathcal{A}$ consists of the space $\mathcal{A}^\mathbb{Z}$ endowed
with the product topology and the \emph{shift map} $\sigma \colon
\mathcal{A}^\mathbb{Z} \to \mathcal{A}^\mathbb{Z}$ defined  by $\sigma(x)_i = x_{i+1}$ for all $i \in
\mathbb{Z}$. Let $\mathcal{A}^*$ be the collection of finite words (also known as
blocks) over $\mathcal{A}$.
For $w \in \mathcal{A}^*$, $\vert w \vert$ will denote the length of $w$.
 A
subset $X \subseteq \mathcal{A}^\mathbb{Z}$ is called a \emph{shift space} if it
is invariant under the shift map and closed. 
For each $\mathcal{F} \subseteq \mathcal{A}^*$, define  $\mathsf{X}_\mathcal{F}$ to be the set of
bi--infinite sequences in $\mathcal{A}^\mathbb{Z}$ which do not contain any of the
\emph{forbidden words} from $\mathcal{F}$. 
A subset $X \subseteq \mathcal{A}^\mathbb{Z}$ is a shift space if and only if there
exists $\mathcal{F} \subseteq \mathcal{A}^*$ such that $X = \mathsf{X}_\mathcal{F}$
(cf. \cite[Proposition 1.3.4]{RJ_lind_marcus}). $X$ is said to be 
a \emph{shift of finite type} (SFT) if this is possible for a finite
set $\mathcal{F}$.

\index{language}
The \emph{language} of a shift space $X$ is denoted $\mathcal{B}(X)$ and it is defined to be the set of all words which occur in
at least one $x \in X$. The shift space $X$ is said to be
\emph{irreducible} if there for every $u,w \in \mathcal{B}(X)$ exists $v \in
\mathcal{B}(X)$ such that $uvw \in \mathcal{B}(X)$. 
For each $x \in X$ define the \emph{left--ray} of $x$ to be $x^- = \cdots
x_{-2} x_{-1}$ and define the \emph{right--ray} of $x$ to be $x^+ = x_0
x_1 x_2 \cdots$. The sets of all left--rays and all right--rays are,
respectively, denoted  $X^-$ and $X^+$. Given a word or ray $x$, $\RJrl(x)$ and $\RJleftl(x)$ will denote respectively the right--most and the left--most letter of $x$.

\index{graph}
\index{path}
\index{edge shift}
\index{irreducible}
A \emph{directed graph} is a quadruple $E = (E^0,E^1,r,s)$ consisting of countable sets $E^0$ and $E^1$, and maps $r,s \colon E^1 \to E^0$.
A \emph{path} $\lambda = e_1 \cdots
e_n$ is a sequence of edges such that $r(e_i) = s(e_{i+1})$ for all $i
\in \{1, \ldots n-1 \}$. The vertices in $E^0$ are considered to be
paths of length $0$.
For each $n \in \mathbb{N}_0$, the set of paths of length $n$ is denoted
$E^n$, and the set of all finite paths is denoted $E^*$.
Extend
the maps $r$ and $s$ to $E^*$ by defining $s(e_1 \cdots e_n) = s(e_1)$
and $r(e_1 \cdots e_n) = r(e_n)$. A directed graph $E$ is said to be
\emph{irreducible} (or transitive) if there for each pair of vertices
$u,v \in E^0$ exists a path $\lambda \in E^*$ with $s(\lambda) = u$
and $r(\lambda) = v$. For a directed graph $E$, the 
\emph{edge shift} $(\mathsf{X}_E, \sigma_E)$ is defined by
$
  \mathsf{X}_E = \left\{ x \in (E^1)^\mathbb{Z} \mid r(x_i) = s(x_{i+1}) \textrm{ for all }
  i \in \mathbb{Z} \right\}$.

\index{conjugacy}
\index{flow equivalence}
\index{Bowen--Franks group}
\index{BF@$\RJBF$|see{Bowen--Franks group}}
\index{Bowen--Franks invariant}
A bijective, continuous and shift commuting map between two shift
spaces is called a \emph{conjugacy}, and when such a map exists, 
the two shift spaces are said to be \emph{conjugate}. \emph{Flow
  equivalence} is a weaker equivalence relation generated by conjugacy
and \emph{symbol expansion} \cite{RJ_parry_sullivan}.
Let $A$ be the adjacency matrix of a directed graph $E$, then $\RJBF(A) = \mathbb{Z}^n / \mathbb{Z}^n (\RJId - A)$ is called the \emph{Bowen--Franks group} of $A$ and it is an invariant of conjugacy of edge shifts. Let $E$ and $F$ be finite directed graphs for which the edge shifts $\mathsf{X}_E$ and $\mathsf{X}_F$ are irreducible and not flow equivalent to the trivial shift with one element, and let $A_E$ and $A_F$ be the corresponding adjacency matrices. Then $\mathsf{X}_E$ and $\mathsf{X}_F$ are flow equivalent if and only $\RJBF(A_E) \simeq \RJBF(A_F)$ and the signs $\RJsgn \det A_E$ and $\RJsgn \det A_F$ are equal \cite{RJ_franks}. Every SFT is conjugate to an edge shift, so this gives a complete flow equivalence invariant of irreducible SFTs. The pair consisting of the Bowen--Franks group and the sign of the determinant is called the \emph{signed Bowen--Franks group}, and it is denoted $\RJBF_+$. This invariant is easy to compute and easy to compare which makes it appealing to consider flow equivalence rather than conjugacy.

\index{graph!labelled}
\index{shift space!presentation of}
\index{presentation}
\index{presentation!follower separated}
A \emph{labelled graph} $(E, \mathcal{L})$ over an alphabet $\mathcal{A}$ consists
of a directed graph $E$ and a surjective labelling map $\mathcal{L} \colon E^1
\to \mathcal{A}$.
Given a labelled graph $(E, \mathcal{L})$, define the shift space $(\mathsf{X}_{(E, \mathcal{L})}, \sigma)$ by setting 
$\mathsf{X}_{(E, \mathcal{L})} = \left\{ \left( \mathcal{L}(x_i) \right)_i \in \mathcal{A}^\mathbb{Z} \mid 
                          x \in \mathsf{X}_E  \right\}$,
The
labelled graph $(E, \mathcal{L})$ is said to be a \emph{presentation} of the
shift space $\mathsf{X}_{(E, \mathcal{L})}$, and a \emph{representative} of a word $w \in
\mathcal{B}(\mathsf{X}_{(E, \mathcal{L})}) $ is a path $\lambda \in E^*$ such that
$\mathcal{L}(\lambda) = w$ with the natural extension of $\mathcal{L}$. Representatives of rays are defined analogously.
Let $(E, \mathcal{L})$ be a labelled graph presenting $X$. For
each $v \in E^0$, define the \emph{predecessor set} of $v$ to be the
set of left--rays in $X$ which have a presentation terminating at $v$. This
is denoted $P_\infty^E(v)$, or just $P_\infty(v)$ when $(E, \mathcal{L})$ is understood from the context.
The presentation $(E, \mathcal{L})$ is said to be \emph{predecessor--separated} if $P_\infty^E(u) \neq P_\infty^E(v)$ when $u,v \in E^0$ and $u \neq v$.

A function $\pi \colon X_1 \to X_2$ between shift spaces $X_1$ and
$X_2$ is said to be a \emph{factor map} if it is continuous,
surjective, and shift commuting. A shift space is
called \emph{sofic} \cite{RJ_weiss} if it is the image of an SFT under a factor
map.
Every SFT is sofic, and a sofic shift
which is not an SFT is called \emph{strictly sofic}.
Fischer proved that a shift space is sofic if and only 
if it can be presented by a finite labelled graph \cite{RJ_fischer}. 
A sofic shift
space is irreducible if and only if it can be presented by an
irreducible labelled graph (see \cite[Sec.\ 3.1]{RJ_lind_marcus}).

Let $(E, \mathcal{L})$ be a finite labelled graph which presents the sofic
shift space $\mathsf{X}_{(E, \mathcal{L})}$, and let $\pi_\mathcal{L} \colon \mathsf{X}_E \to \mathsf{X}_{(E,
  \mathcal{L})}$ be the factor map induced by the labelling map $\mathcal{L} \colon
E^1 \to \mathcal{A}$, then the SFT $\mathsf{X}_E$ is called a \emph{cover} of the
sofic shift $\mathsf{X}_{(E, \mathcal{L})}$, and $\pi_\mathcal{L}$ is called the covering map. 

Let $X$ be a shift space over an alphabet $\mathcal{A}$. A presentation $(E,\mathcal{L})$ of $X$ is said to be \emph{left--resolving}
if no vertex in $E^0$ receives two edges with the same label. 
Fischer proved \cite{RJ_fischer} that up to labelled graph
isomorphism every irreducible sofic shift has a unique left--resolving
presentation with fewer vertices than any other left--resolving
presentation. This is called the \emph{left Fischer cover} of $X$,
and it is denoted $(F, \mathcal{L}_F)$. 

\index{predecessor set}
\index{follower set}
\index{intrinsically synchronising}
For $x^+ \in X^+$, define the \emph{predecessor set} of $x^+$ to be
the set of left--rays which may  precede $x^+$ in $X$, that is
$P_\infty(x^+) = \{ y^- \in X^- \mid y^- x^+ \in X \}$ 
(see \cite[Secs. I and III]{RJ_jonoska_marcus} and \cite[Exercise
    3.2.8]{RJ_lind_marcus} for details). The \emph{follower set} of a left--ray $x^- \in X^-$ is defined analogously. 
The \emph{left Krieger cover} of the sofic shift space $X$ is the labelled graph
$(K, \mathcal{L}_K)$ where $K^0 = \{ P_\infty(x^+) \mid x^+ \in X^+\}$,
and where there is an edge labelled $a \in \mathcal{A}$ from $P \in K^0$ to
$P' \in K^0$ if and only if there exists $x^+ \in X^+$ such that $P
= P_\infty(a x^+)$ and $P' = P_\infty(x^+)$.
A word $v  \in \mathcal B(X)$ is said to be \emph{intrinsically synchronising} if $uvw \in \mathcal B(X)$ whenever $u$ and $w$ are words such that $uv, vw \in \mathcal B(X )$.
A ray is said to be \emph{intrinsically synchronising} if it contains an intrinsically synchronising word as a factor.
If a right--ray $x^+$ is intrinsically synchronising, then there is precisely one vertex in the left Fischer cover where a presentation of $x^+$ can start, and this vertex can be identified with the predecessor set $P_\infty(x^+)$ as a vertex in the Krieger cover. In this way, the left Fischer cover can be identified with the irreducible component of the left Krieger cover generated by the vertices that are predecessor sets of intrinsically synchronising right--rays \cite[Lemma 2.7]{RJ_krieger_sofic_I}, \cite[Exercise 3.3.4]{RJ_lind_marcus}. The interplay between the structure of the Fischer and Krieger covers is examined in detail in \cite{RJ_johansen_structure}.

\index{renewal system}
\index{renewal system!generating list of|see{generating list}}
\index{generating list}\index{XL@$\mathsf{X}(L)$}
\index{renewal system!loop graph of}
\index{flower automata}\index{loop system}
Let $\mathcal{A}$ be an alphabet, let $L \subseteq \mathcal{A}^*$ be a finite list of words over $\mathcal{A}$, and define $\mathcal{B}(L)$ to be the set of factors of elements of $L^*$. Then $\mathcal{B}(L)$ is the language of a shift space $\mathsf{X}(L)$ which is said to be the \emph{renewal system}
generated by $L$. $L$ is said to be the \emph{generating list} of $\mathsf{X}(L)$. 
A renewal system is an irreducible sofic shift since it can be presented by the labelled graph obtained by writing the generating words on loops starting and ending at a common vertex. 
This graph is called the \emph{standard loop graph presentation} of $\mathsf{X}(L)$, and because of this presentation, renewal systems are called \emph{loop systems} or \emph{flower automata} in automata theory (e.g.\  \cite{RJ_berstel_perrin}). 

\index{Adler's problem}
\index{Adler's problem!flow equivalence version}
Simple examples show that not every sofic shift---or every SFT---is a renewal system \cite[pp.\ 433]{RJ_lind_marcus}, and these results naturally raise the following question, which was first asked by Adler:
Is every irreducible shift of finite type conjugate to a renewal system?
This question has been the motivation of most of the work done on renewal systems \cite{RJ_goldberger_lind_smorodinsky,RJ_hong_shin_cyclic,RJ_hong_shin,RJ_johnson_madden,RJ_restivo, RJ_restivo_note,RJ_williams_rs}.
The analogous question for sofic shifts has a negative answer \cite{RJ_williams_rs}. The aim of the present work has been to answer another natural variation of Adler's question: Is every irreducible SFT \emph{flow equivalent} to a renewal system? 
To answer this question, it is sufficient to find the range of the Bowen--Franks invariant over the set of SFT renewal systems and check whether it is equal to the range over the set of irreducible SFTs.
It is easy to check that a group $G$ is the Bowen--Franks group of an irreducible SFT if and only if it is a finitely generated abelian group and that any combination of sign and Bowen--Franks group can be achieved by the Bowen--Franks invariant.
Hence, the overall strategy of the investigation of the flow equivalence question has been to attempt to construct all these combinations of groups and signs. 
However, it is difficult to construct renewal systems attaining many of the values of the invariant. In fact, it is non--trivial to construct an SFT renewal system that is not flow equivalent to a full shift \cite{RJ_johansen_thesis}.

\section{Fischer covers of renewal systems}
\label{RJ_sec_rs_lfc}
In the attempt to find the range of the Bowen--Franks invariant over the set of SFT renewal systems, it is useful to be able to construct complicated renewal systems from simpler building blocks, but in general, it is non--trivial to study the structure of the renewal system $\mathsf{X}(L_1 \cup L_2)$ even if the renewal systems $\mathsf{X}(L_1)$ and $\mathsf{X}(L_2)$ are well understood.
The goal of this section is to describe the structure of the left Fischer covers of renewal systems in order to give conditions under which the Fischer cover of $\mathsf{X}(L_1 \cup L_2)$ can be constructed when the Fischer covers of $\mathsf{X}(L_1)$ and $\mathsf{X}(L_2)$ are known.

\index{renewal system!Fischer cover of!$P_0(L)$}\index{P0@$P_0(L)$}
\index{renewal system!Fischer cover of!construction} 
\label{RJ_sec_rs_lfc_construction}

Let $L$ be a generating list and define 
$  P_0(L) = \left\{  \ldots w_{-2} w_{-1} w_{0}  \mid w_i \in L  \right \} \subseteq \mathsf{X}(L)^-$.
$P_0(L)$ is the predecessor set of the central vertex in the standard loop graph of $\mathsf{X}(L)$, but it is not necessarily the predecessor set of a right--ray in $\mathsf{X}(L)^+$, so it does not necessarily correspond to a vertex in the left Fischer cover of $\mathsf{X}(L)$. If $p \in \mathcal{B}(\mathsf{X}(L))$ is a prefix of some word in $L$, define
$	P_0(L)p = \left\{  \ldots w_{-2} w_{-1} w_{0} p  \mid w_i \in L  \right \} \subseteq \mathsf{X}(L)^-$.



\index{partitioning}\index{partitioning!beginning of}\index{partitioning!end of}
Let $L$ be a generating list.
A triple $(n_b,g,l)$ where $n_b,l \in \mathbb{N}$ and $g$ is an ordered list of words $g_1, \ldots , g_k \in L$ with $\sum_{i=1}^k \lvert g_i \rvert \geq  n_b+l-1$ is said to be a \emph{partitioning} of the factor $v_{[n_b,n_b+l-1]} \in \mathcal{B}(\mathsf{X}(L))$ of $v = g_1 \cdots g_k$. The \emph{beginning} of the partitioning is the word $v_{[1,n_b-1]}$, and the \emph{end} is the word $v_{[n_b+l,\lvert v \rvert]}$.
A partitioning of a right--ray $x^+ \in \mathsf{X}(L)^+$ is a pair $p = (n_b,(g_i)_{i\in\mathbb{N}})$ where $n_b \in \mathbb{N}$ and $g_i \in L$ such that $wx^+ = g_1 g_2 \cdots$ when $w$ is the \emph{beginning} consisting of the $n_b-1$ first letters of the concatenation $g_1 g_2 \cdots$.
Partitionings of left--rays are defined analogously.

\index{bordering}
\index{bordering!strongly}
\index{word!bordering}
\index{word!bordering!strongly}
\index{generating list!irreducible}
Let $L \subseteq \mathcal{A}^*$ be a finite list, and let $w \in \mathcal{B}(\mathsf{X}(L)) \cup \mathsf{X}(L)^+$ be an allowed word or right--ray. Then $w$ is said to be \emph{left--bordering} if there exists a partitioning of $w$ with empty beginning, and \emph{strongly left--bordering} if every partitioning of $w$ has empty beginning. Right--bordering words and left--rays are defined analogously.


\begin{definition}	\label{RJ_def_border_point}
\index{border point}\index{border point!universal}
\index{border point!generator of}
Let $L \subseteq \mathcal{A}^*$ be finite, and let $(F, \mathcal{L}_F)$ be the left Fischer cover of $\mathsf{X}(L)$. A vertex $P \in F^0$ is said to be a \emph{(universal) border point} for $L$ if there exists a (strongly) left--bordering $x^+ \in X^+$ such that $P = P_\infty(x^+)$. An intrinsically synchronizing word $w \in L^*$ is said to be a \emph{generator} of the border point $P_\infty(w) = P_\infty(w^\infty)$, and it is said to be a \emph{minimal generator} of $P$ if no prefix of $w$ is a generator of $P$.
\end{definition}

\index{renewal system!border point of}
The border points add information to the Fischer cover about the structure of the generating lists, and this information will be useful for studying $\mathsf{X}(L_1 \cup L_2)$ when the Fischer covers of $\mathsf{X}(L_1)$ and $\mathsf{X}(L_2)$ are known. If $P$ is a (universal) border point of $L$ and there is no ambiguity about which list is generating $X = \mathsf{X}(L)$, then the terminology will be abused slightly by saying that $P$ is a (universal) border point of $X$ or simply of the left Fischer cover.

\begin{lemma}
\label{RJ_lem_border_point_general_prop}
Let $L$ be a finite list generating a renewal system with left Fischer cover $(F, \mathcal{L}_F)$.
\begin{enumerate}
\item \label{RJ_lem_bp_subset}
If $P \in F^0$ is a border point, then $P_0(L) \subseteq P$, and if $P$ is universal then $P = P_0(L)$.
\item \label{RJ_lem_bp_path_to}
If $P_1,P_2 \in F^0$ are border points and if $w_1 \in L^*$ is a generator of $P_1$, then there exists a path with label $w_1$ from $P_1$ to $P_2$. 
\item \label{RJ_lem_bp_path_implies}
If $P_1 \in F^0$ is a border point and $w \in L^*$, then there exists a unique border point $P_2 \in F^0$ with a path labelled $w$ from $P_2$ to $P_1$.
\item 
If $\mathsf{X}(L)$ is an SFT, then every border point of $L$ has a generator.
\item \label{RJ_lem_bp_strongly_right_implies}
If $L$ has a strongly right--bordering word $w$, then $x^+ \in \mathsf{X}(L)^+$ is left--bordering if and only if $P_\infty(x^+)$ is a border point.
\end{enumerate}
\end{lemma} 

\begin{proof} 
1.
Choose a left--bordering $x^+ \in \mathsf{X}(L)^+$ such that $P = P_\infty(x^+)$ and note that $y^-x^+ \in \mathsf{X}(L)$ for each $y^- \in P_0(L)$.
2.
Choose a left--bordering $x^+ \in \mathsf{X}(L)^+$ such that $P_2 = P_\infty(x^+)$. Then $P_\infty(w_1 x^+) = P_1$ since $w_1 x^+ \in \mathsf{X}(L)^+$ and  $w_1$ is intrinsically synchronizing, so there is a path labelled $w_1$ from $P_1$ to $P_2$.
3.
Choose a left--bordering $x^+ \in \mathsf{X}(L)^+$ such that $P = P_\infty(x^+)$. Since $w \in L^*$, the right--ray $w x^+$ is also left--bordering.
4.
Let $P = P_\infty(x^+)$ for some left--bordering $x^+ \in \mathsf{X}(L)^+$, and choose an intrinsically synchronizing prefix $w \in L^*$ of $x^+$. Then $P_\infty(x^+) = P_\infty(w)$, so $w$ is a generator of $P$.
5. 
If $P_\infty(x^+)$ is a border point, then $w x^+ \in \mathsf{X}(L)^+$, so $x^+$ must be left--bordering. The other implication holds by definition.
\end{proof}

\noindent
In particular, the universal border point is unique when it exists.
A predecessor set $P_\infty(x^+)$ can be a border point even though $x^+$ is not left--bordering 

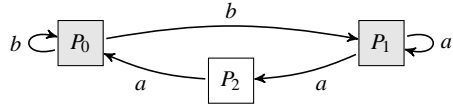
\begin{figure}[htbp]
\begin{tikzpicture}
  [bend angle=10,
   clearRound/.style = {circle, inner sep = 0pt, minimum size = 17mm},
   clear/.style = {rectangle, minimum width = 5 mm, minimum height = 5 mm, inner sep = 0pt},  
   greyRound/.style = {circle, draw, minimum size = 1 mm, inner sep =
      0pt, fill=black!10},
   grey/.style = {rectangle, draw, minimum size = 6 mm, inner sep =
      1pt, fill=black!10},
    white/.style = {rectangle, draw, minimum size = 6 mm, inner sep =
      1pt},
   to/.style = {->, shorten <= 1 pt, >=stealth', semithick}]
  
  \node[grey] (P0) at (0,0) {$P_0$};
  \node[grey] (P1) at (4,0) {$P_1$};
  \node[white] (P2) at (2,-0.5) {$P_2$};
  \node (space) at (6,0) {};
  
  \draw[to,loop left] (P0) to node[auto] {$b$} (P0);
  \draw[to,bend left] (P0) to node[auto] {$b$} (P1);
  \draw[to,bend left] (P1) to node[auto] {$a$} (P2);
  \draw[to,bend left] (P2) to node[auto] {$a$} (P0);
  \draw[to,loop right] (P1) to node[auto] {$a$} (P1);

\end{tikzpicture}
\hspace{\stretch{1}}
\sidecaption
\caption[Border points.]{Left Fischer cover of the SFT renewal system $\mathsf{X}(L)$ generated by $L = \{ aa, aaa, b \}$ discussed in Example \ref{RJ_ex_border_points}. The border points are coloured grey.
} 
\label{RJ_fig_aa_aaa_b}
\end{figure}

\begin{example} \label{RJ_ex_border_points}
Consider the list $L = \{ aa, aaa, b\}$ and the renewal system $\mathsf{X}(L)$. It is straightforward to check that $\mathsf{X}(L) = \mathsf{X}_\mathcal{F}$ for the set of forbidden words $\mathcal{F} = \{ bab \}$, so this is an SFT. For this shift, there are three distinct predecessor sets:
\begin{align*}
P_0 &= P_\infty(b\cdots) = \{ \cdots x_{-1}x_{0}  \in \mathsf{X}(L)^- \mid x_0 = b \textnormal{ or }  x_{-1}x_{0} = aa \}, \\
P_1 &= P_\infty(a^nb\cdots) = P_\infty(a^\infty) =  \mathsf{X}(L)^- , \qquad n \geq 2, \\
P_2 &= P_\infty(ab\cdots) = \{ \cdots x_{-1}x_{0} \in \mathsf{X}(L)^- \mid x_{0} = a  \}.
\end{align*}
The information contained in these equations is sufficient to draw the left Krieger cover, and each set is the predecessor set of an intrinsically synchronising right--ray, so the left Fischer cover can be identified with the left Krieger cover. This graph is shown in Fig.\ \ref{RJ_fig_aa_aaa_b}.
Here, $P_0$ is a universal border point because any right--ray starting with a $b$ is strongly left bordering. The generating word $b$ is a minimal generator of $P_0$. The vertex $P_1$ is a border point because $a^nb\cdots$ is left bordering for all $n \geq 2$.  
The word $aa$ is a minimal generator of $P_1$, and $aab$ is a non--minimal generator. The vertex $P_2$ is not a border point since there is no infinite concatenation $x^+$ of words from $L$ such that $x^+= ab\cdots$.  
Another way to see this is to note that every path terminating at $P_2$ has $a$ as a suffix, so that $P_0$ is not a subset of $P_2$ which together with Lemma \ref{RJ_lem_border_point_general_prop}(\ref{RJ_lem_bp_subset}) implies that $P_2$ is not a border point.
Note also that Lemma \ref{RJ_lem_border_point_general_prop}(\ref{RJ_lem_bp_path_to}) means that there must be paths labelled $b$ from $P_0$ to the two border points, and similarly, paths labelled $aa$ and $aab$ from $P_1$ to the two border points.
\end{example}


\index{renewal system!addition of}
Consider two renewal systems $\mathsf{X}(L_1)$ and $\mathsf{X}(L_2)$. The \emph{sum} $\mathsf{X}(L_1) + \mathsf{X}(L_2)$ is the renewal system $\mathsf{X}(L_1 \cup L_2)$.
Generally, it is non--trivial to construct the Fischer cover of such a sum even if the Fischer covers of the summands are known.

%


\index{generating list!modular}
\begin{definition}
Let $L$ be a generating list with universal border point $P_0$ and let $(F, \mathcal{L}_F)$ be the left Fischer cover of $\mathsf{X}(L)$.  
$L$ is said to be \emph{left--modular} if for all $\lambda \in F^*$ with 
$r(\lambda) = P_0$, $\mathcal{L}_F(\lambda) \in L^*$ if and only if $s(\lambda)$ is a border point. \emph{Right--modular} generating lists are defined analogously.
\end{definition}

\index{renewal system!modular}
\noindent
It is straightforward to check that the list considered in Example \ref{RJ_ex_border_points} is left--modular.
When $L$ is left--modular and there is no doubt about which generating list is used, the renewal system $\mathsf{X}(L)$ will also be said to be \emph{left--modular}.

\begin{lemma} \label{RJ_lem_strongly_bordering_implies_modular}
If $L$ is a generating list with a strongly left--bordering word $w_l$ and a strongly right--bordering word $w_r$, then it is both left-- and right--modular.
\end{lemma}

\begin{proof}
Let $(F, \mathcal{L}_F)$ be the left Fischer cover of $\mathsf{X}(L)$, let $P \in F^0$ be a border point, and choose $x^+ \in \mathsf{X}(L)^+$ such that $w_lx^+ \in \mathsf{X}(L)^+$. Assume that there is a path from $P$ to $P_0(L) = P_\infty(w_lx^+)$ with label $w$. The word $w_r$ has a partitioning with empty end, so there is a path labelled $w_r$ terminating at $P$. It follows that $w_r w w_l x^+ \in \mathsf{X}(L)^+$, so $w \in L^*$. By symmetry, $L$ is also right--modular.
\end{proof}


For $i \in \{1,2\}$, let $L_i$ be a left--modular generating list and let $X_i = \mathsf{X}(L_i)$ have alphabet $\mathcal{A}_i$ and left Fischer cover $(F_i, \mathcal{L}_i)$. Let $P_i \in F_i^0$ be the universal border point of $L_i$. Assume that $\mathcal{A}_1 \cap \mathcal{A}_2 = \emptyset$.
The left Fischer cover of $X_1 + X_2$ will turn out to be the labelled graph $(F_+, \mathcal{L}_+)$ obtained by taking the union of $(F_1, \mathcal{L}_1)$ and $(F_2, \mathcal{L}_2)$, identifying the two universal border points $P_1$ and $P_2$, and adding certain connecting edges. 
To do this formally, introduce a new vertex $P_+$ and define
$F_+^0 = ( F_1^0 \cup F_2^0 \cup \{P_+\} ) \setminus \{ P_1, P_2 \}$.
Define maps $f_i \colon F_i^0 \to F_+^0$ such that for $v \in F_i^0 \setminus \{ P_i \}$, $f_i(v)$ is the vertex in $F_+^0$ corresponding to $v$ and such that $f_i(P_i) = P_+$. For each $e \in F_i^1$, define an edge $e' \in F_+^1$ such that $s(e') = f_i(s(e))$, $r(e') = f_i(r(e))$, and $\mathcal{L}_+(e') = \mathcal{L}_i(e)$. For each $e \in F_1^1$ with $r(e) = P_1$ and each non--universal border point $P \in F_2^0$, draw an additional edge $e' \in F_+^1$ with $s(e') = f_1(s(e))$, $r(e') = f_2(P)$, and $\mathcal{L}_+(e') = \mathcal{L}_1(e)$. Draw analogous edges for each $e \in F_2^1$ with $r(e) = P_2$ and every non--universal border point $P \in F_1^0$. 
This construction
is illustrated in Fig.\ \ref{RJ_fig_modular_add}.
 
 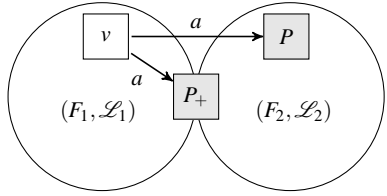
\begin{figure}[htb]
\begin{tikzpicture}
  [bend angle=10,
   clearRound/.style = {circle, inner sep = 0pt, minimum size = 17mm},
   clear/.style = {rectangle, minimum width = 17 mm, minimum height = 6 mm, inner sep = 0pt},  
   greyRound/.style = {circle, draw, minimum size = 1 mm, inner sep =
      0pt, fill=black!10},
   grey/.style = {rectangle, draw, minimum size = 6 mm, inner sep =
      1pt, fill=black!10},
   white/.style = {rectangle, draw, minimum size = 6 mm, inner sep =
      1pt},
   lfc/.style = {circle, draw, minimum size = 2.5cm},
   to/.style = {->, shorten <= 1 pt, >=stealth', semithick}]
  
  \node[lfc] (F1) at (-1.25,0) {};
  \node[lfc] (F2) at (1.25,0) {};
  \node[grey] (P+) at (0,0) {$P_+$};
  \node[white] (S) at (-1.2,0.8) {$v$};
  \node[grey] (P) at (1.2,0.8) {$P$};  
  \node[clear] (F1text) at (-1.3,-0.2) {$(F_1,\mathcal{L}_1)$};
  \node[clear] (F1text) at (1.3,-0.2) {$(F_2,\mathcal{L}_2)$};   
  \node (space) at (5,0) {};
  
  \draw[to] (S) to node[auto,swap] {$a$} (P+);
  \draw[to] (S) to node[auto] {$a$} (P);

\end{tikzpicture}
\hspace{\stretch{1}}
\sidecaption
\caption[Addition of modular renewal systems.]{The labelled graph $(F_+, \mathcal{L}_+)$. In $(F_1, \mathcal{L}_1)$, $v$ emits an edge labelled $a$ to $P_1$, so in $(F_+, \mathcal{L}_+)$, the corresponding vertex emits edges labelled $a$ to every vertex corresponding to a border point $P \in F_2^0$.} 
\label{RJ_fig_modular_add}
\end{figure}

\begin{proposition}
\label{RJ_prop_addition_modular}
If $L_1$ and $L_2$ are left--modular generating lists with disjoint alphabets, then $L_1 \cup L_2$ is left--modular, the left Fischer cover of $\mathsf{X}(L_1 \cup L_2)$ is the graph $(F_+, \mathcal{L}_+)$ constructed above, and the vertex $P_+ \in F_+^0$ is the universal border point of $L_1 \cup L_2$. 
\end{proposition}

\begin{proof}
By construction, the labelled graph $(F_+, \mathcal{L}_+)$ is irreducible, left--resolving, and predecessor--separated, so it is the left Fischer cover of some sofic shift $X_+$ \cite[Cor. 3.3.19]{RJ_lind_marcus}.
Given $w \in L_1^*$, there is a path with label $w$ in the left Fischer cover of $X_1$ from some border point $P \in F_1^0$ to the universal border point $P_1$ by Lemma \ref{RJ_lem_border_point_general_prop}(\ref{RJ_lem_bp_path_implies}). Hence, there is also a path labelled $w$ in $(F_+, \mathcal{L}_+)$ from the vertex corresponding to $P$ to the vertex $P_+$. This means that for every border point $Q \in F_2^0$, $(F_+, \mathcal{L}_+)$ contains a path labelled $w$ from the vertex corresponding to $P$ to the vertex corresponding to $Q$. By symmetry, it follows that every element of $(L_X \cup  L_Y)^*$ has a presentation in $(F_+, \mathcal{L}_+)$. Hence, $\mathsf{X}(L_1 \cup L_2) \subset X_+$.

Assume that $awb \in \mathcal{B}(X_+)$ with $a,b \in \mathcal{A}_1$ and $w \in \mathcal{A}_2^*$. Then there must be a path labelled $w$ in $(F_+, \mathcal{L}_+)$ from a vertex corresponding to a border point $P$ of $L_2$ to $P_+$. By construction, this is only possible if there is also a path labelled $w$ from $P$ to $P_2$ in $(F_2, \mathcal{L}_2)$, but $L_2$ is left--modular, so this means that $w \in L_2^*$. By symmetry, $\mathsf{X}(L_1 \cup L_2) = X_+$, and $P_+$ is the universal border point by construction.
\end{proof} 

\index{fragmentation}
Let $X$ be a shift space over the alphabet $\mathcal{A}$. Given $a \in \mathcal{A}$, $k \in \mathbb{N}$, and new symbols $a_1, \ldots, a_k \notin \mathcal{A}$  consider the map 
$f_{a,k} \colon (\mathcal{A} \setminus \{ a \}) \cup \{ a_1, \ldots , a_k\} \to \mathcal{A}$ defined by
$f_{a,k}(a_i) = a$ for each $1 \leq i \leq k$ and $f_{a,k}(b) = b$ when $b \in \mathcal{A} \setminus \{a \}$. Let $F_{a,k} \colon ((\mathcal{A} \setminus \{ a \}) \cup \{ a_1, \ldots , a_k\})^* \to \mathcal{A}^*$ be the natural extension of $f_{a,k}$.  
If $w \in \mathcal{A}^*$ contains $l$ copies of the symbol $a$, then the preimage $F_{a,k}^{-1}(\{w \})$ is the set consisting of the $k^l$ words that can be obtained by replacing the $a$s by the symbols $a_1, \ldots, a_k$.

\begin{definition} 
\index{renewal system!fragmentation of}
\index{shift space!fragmentation of}
\index{X*ak@$X_{a,k}$}
\index{F*ak@$F_{a,k}$}
Let $X = \mathsf{X}_\mathcal{F}$ be a shift space over the alphabet $\mathcal{A}$, let $a \in \mathcal{A}$, let $a_1, \ldots, a_k \notin \mathcal{A}$, and let $F_{a,k}$ be defined as above. Then the shift space $X_{a,k} = \mathsf{X}_{F_{a,k}^{-1}(\mathcal{F})}$ is said to be the shift obtained from $X$ by \emph{fragmenting} $a$ into $a_1, \ldots, a_k$.
\end{definition}

\index{L*@$L_{a,k}$}
\noindent
Note that this construction does not depend on the choice of $\mathcal F$ representing $X$, in particular,  $\mathcal{B}( X_{a,k} ) = F_{a,k}^{-1}(\mathcal{B}(X))$. Furthermore, $X_{a,k}$ is an SFT if and only if $X$ is an SFT.
If $X$ is an irreducible sofic shift, then the left and right Fischer and Krieger covers of $X_{a,k}$ are obtained by replacing each edge labelled $a$ in the corresponding cover of $X$ by $k$ edges labelled $a_1, \ldots, a_k$.
Note that $X$ and $X_{a,k}$ are not generally conjugate or even flow equivalent.
If $X = \mathsf{X}(L)$ is a renewal system, then $X_{a,k}$ is the renewal system generated by the list $L_{a,k} = F_{a,k}^{-1}(L)$. 

\begin{remark}
\label{RJ_rem_fragmentation}
\label{RJ_rem_sum_frag_commute}
Let $A$ be the symbolic adjacency matrix of the left Fischer cover of an SFT renewal system $\mathsf{X}(L)$ with alphabet $\mathcal{A}$. Given $a \in \mathcal{A}$ and $k \in \mathbb{N}$, define $f \colon \mathcal{A} \to \mathbb{N}$ by $f(a) = k$ and $f(b) = 1$ for $b \neq a$. Extend $f$ to the set of finite formal sums over $\mathcal{A}$ in the natural way and consider the integer matrix $f(A)$. Then $f(A)$ is the adjacency matrix of the underlying graph of the left Fischer cover of $\mathsf{X}(L_{a,k})$.
For lists over disjoint alphabets, it follows immediately from the definitions that fragmentation and addition commute. 
\end{remark}

\section{Entropy and flow equivalence}
\label{RJ_sec_rs_entropy}
Hong and Shin \cite{RJ_hong_shin} have constructed a class $H$ of lists generating SFT renewal systems such that $\log \lambda$ is the entropy of an SFT if and only if there exists $L \in H$ with $h(\mathsf{X}(L)) = \log \lambda$, and this is arguably the most powerful general result known about the invariants of SFT renewal systems. In the following, the renewal systems generated by lists from $H$ will be classified up to flow equivalence.
As demonstrated in \cite{RJ_johansen_thesis}, it is difficult to construct renewal systems with non--cyclic Bowen--Franks groups and/or positive determinants directly, and this classification will yield hitherto unseen values of the invariant.


The construction of the class $H$ of generating lists considered in \cite{RJ_hong_shin} will be modified slightly since some of the details of the original construction are invisible up to flow equivalence. In particular, several words from the generating lists can be replaced by single symbols by using symbol reduction. Additionally, there are extra conditions on some of the variables in \cite{RJ_hong_shin} which will be omitted here since the larger class can be classified without extra work.

\index{generating list!in $B$}
\index{B*@$B$}
Let $r \geq 2$ and let $n_1, \ldots, n_r,c_1, \ldots , c_r,d, N \in \mathbb{N}$, and let $W$ be the set consisting of the following words:
\begin{itemize}
\item $\alpha_i = \alpha_{i,1} \cdots \alpha_{i,n_1}$ for $1 \leq i \leq c_1$ 
\item $\tilde \alpha_i = \tilde \alpha_{i,1} \cdots \tilde \alpha_{i, n_1}$ for $1 \leq i \leq c_1$
\item $\gamma_{k,i_k} = \gamma_{k,i_k,1} \cdots \gamma_{k,i_k,n_k}$ for $2 \leq k \leq r$ and $1 \leq i_k \leq c_k$
\item $\alpha_{i_1} \gamma_{2,i_2} \cdots \gamma_{r,i_r} \beta_l^N$
         for $1 \leq i_j \leq c_j$ and $1 \leq l \leq d$
\item $\beta_l^N \tilde \alpha_{i_1} \gamma_{2,i_2} \cdots \gamma_{r,i_r}$
         for $1 \leq i_j \leq c_j$ and $1 \leq l \leq d$.
\end{itemize}
The set of generating lists of this form will be denoted $B$.

\begin{remark}
\label{RJ_rem_R_def}
\index{generating list!in $R$}
\index{R*@$R$}
Symbol reduction can be used to reduce the words $\alpha_i$, $\tilde \alpha_i$, $\gamma_{k,i_k}$, and $\beta_l^N$ to single letters \cite[Lemmas 2.15 and 2.23]{RJ_johansen_thesis}, so up to flow equivalence, the list $W \in B$ considered above can be replaced by the list $W'$ consisting of the one--letter words $\alpha_i$, $\tilde \alpha_i$, and $\gamma_{k,i}$ as well as the words
\begin{itemize}
\item $\alpha_{i_1} \gamma_{2,i_2} \cdots \gamma_{r,i_r} \beta_l$
         for $1 \leq i_j \leq c_j$ and $1 \leq l \leq d$
\item $\beta_l \tilde \alpha_{i_1} \gamma_{2,i_2} \cdots \gamma_{r,i_r}$
         for $1 \leq i_j \leq c_j$ and $1 \leq l \leq d$.
\end{itemize}
Furthermore, if 
\begin{equation} \label{RJ_eq_reduced_list}
L = \{ \alpha, \tilde \alpha, 
 \alpha \gamma_2 \cdots \gamma_r \beta, \beta \tilde \alpha \gamma_2 \cdots \gamma_r \} \cup \{ \gamma_k \mid 2 \leq k \leq r\} \;, 
\end{equation}
then $\mathsf{X}(W')$ can be obtained from $\mathsf{X}(L)$ by fragmenting $\alpha$ to $\alpha_1, \ldots,  \alpha_{c_1}$, $\beta$ to $\beta_1, \ldots , \beta_l$ and so on.
Let $R$ be the set of generating lists of the form given in (\ref{RJ_eq_reduced_list}). 
\end{remark}

\index{H@$H$}\index{generating list!in $H$}
Next consider generating lists $W_1, \ldots , W_m \in B$ with disjoint alphabets, and let $W = \bigcup_{j=1}^m W_j$. 
Let $\tilde W$ be a finite set of words that do not share any letters with each other or with the words from $W$, and consider the generating list $W \cup \tilde W$. Let $\tilde H$ be the set of generating lists that can be constructed in this manner. Let $\mu$ be a Perron  number. Then there exists $\tilde L \in \tilde H$ such that $\mathsf{X}(\tilde L)$ is an SFT and $h(\mathsf{X}(\tilde L)) = \log \mu$ \cite{RJ_hong_shin}.

\begin{remark} \label{RJ_rem_H_letter_frag}
If $W \cup \tilde W \in \tilde H$ as above, then symbol reduction can be used to show that $\mathsf{X}(W \cup \tilde W)$ is flow equivalent to the renewal system generated by the union of $W$ and $\lvert \tilde W \rvert$ new letters \cite[Lemma 2.23]{RJ_johansen_thesis}, i.e.\ $\mathsf{X}(W \cup \tilde W)$ is flow equivalent to a fragmentation of $\mathsf{X}(W \cup \{ a \})$ when $a \notin \mathcal{A}(\mathsf{X}(W))$.
\end{remark}

\index{renewal system!entropy of}
Consider a generating list $\tilde L \in \tilde H$ and $p \in \mathbb{N}$. For each letter $a \in  \mathcal{A}( \mathsf{X}(\tilde L))$, introduce new letters $a_1, \dots , a_p \notin \mathcal{A}( \mathsf{X}(\tilde L))$, and let $L$ denote the generating list obtained by replacing each occurrence of $a$ in $\tilde L$ by the word $a_1\cdots a_p$.
Let $H$ denote the set of generating lists that can be obtained from $\tilde H$ in this manner. Let $\lambda$ be a weak Perron number. Then there exists $L \in H$ such that $\mathsf{X}(L)$ is an SFT and $h(\mathsf{X}(L)) = \log \lambda$ \cite{RJ_hong_shin}.

\begin{remark}
\label{RJ_rem_H_tilde_fe}
If $L$ is obtained from $\tilde L \in \tilde H$ as above, then $\mathsf{X}(L) \sim_{\textrm{FE}} \mathsf{X}(\tilde L)$ since the modification can be achieved using symbol expansion of each $a \in \mathcal{A}( \mathsf{X}(\tilde L))$.
\end{remark}

The next step is to prove that the building blocks in the class $R$ introduced in Remark \ref{RJ_rem_R_def} are left--modular, and to construct the Fischer covers of the corresponding renewal systems. As the following lemmas show, this will allow a classification of the renewal systems generated by lists from $H$ via addition and fragmentation. The first result follows immediately from Remarks \ref{RJ_rem_sum_frag_commute},  \ref{RJ_rem_R_def}, \ref{RJ_rem_H_letter_frag}, and \ref{RJ_rem_H_tilde_fe}.

\begin{lemma}
\label{RJ_lem_entropy_fe}
For each $L \in H$, there exist $L_1, \ldots , L_m \in R$ such that $\mathsf{X}(L)$ is flow equivalent to a fragmentation of $\mathsf{X} ( \bigcup_{j=0}^m L_j )$, where $L_0 = \{ a \}$ for some $a$ that does not occur in $L_1, \ldots , L_m$. 
\end{lemma}


\begin{lemma}
\label{RJ_lem_entropy_lfc}
If $L \in R$, then $L$ is left--modular, $\mathsf{X}(L)$ is an SFT, and the left Fischer cover of $\mathsf{X}(L)$ is the labelled graph shown in Fig.\ \ref{RJ_fig_entropy_lfc}.
\end{lemma}

\begin{proof}
Let
\begin{equation}
\label{RJ_eq_R_L}
L = \{ \alpha, \tilde \alpha,  
 \alpha \gamma_2 \cdots \gamma_r \beta, \beta \tilde \alpha \gamma_2 \cdots \gamma_r \} \cup \{ \gamma_k \mid 2 \leq k \leq r\} \in R \;.
\end{equation}
The word $\alpha \gamma_2 \cdots \gamma_r \beta \beta \tilde \alpha \gamma_2 \cdots \gamma_r$ is strongly left-- and right--bordering, so $L$ is left-- and right--modular by Lemma \ref{RJ_lem_strongly_bordering_implies_modular}. Let $P_0 = P_0(L)$.
If $x^+ \in \mathsf{X}(L)^+$ does not have a suffix of a product of the generating words $\alpha \gamma_2 \cdots \gamma_r \beta$ and $\beta \tilde \alpha \gamma_2 \cdots \gamma_r$ as a prefix, then $x^+$ is strongly left--bordering, so $P_\infty(x^+) = P_0$. Hence, to determine the rest of the predecessor sets and thereby the vertices of the left Fischer cover, it is sufficient to consider right--rays that do have such a prefix.

Consider first $x^+ \in \mathsf{X}(L)^+$ such that $\beta x^+ \in \mathsf{X}(L)^+$. The letter $\beta$ must come from either $\alpha \gamma_2 \cdots \gamma_r \beta$ or $\beta \tilde \alpha \gamma_2 \cdots \gamma_r$, so the beginning of a partitioning of $\beta x^+$ must be either empty or equal to $\alpha \gamma_2 \cdots \gamma_r$.
Assume first that every partitioning of $\beta x^+$ has beginning $\alpha \gamma_2 \cdots \gamma_r$ (i.e.\ that $\tilde \alpha \gamma_2 \cdots \gamma_r$ is not a prefix of $x^+$).
In this case, $\beta x^+$ must be preceded by $\alpha \gamma_2 \cdots \gamma_r$, and the corresponding predecessor sets are: 
\begin{align}
\label{RJ_eq_entropy_predecessor_sets}
  P_\infty( \alpha \gamma_2 \cdots \gamma_r \beta x^+ ) &= P_0 \nonumber \\ 
  P_\infty( \gamma_2 \cdots \gamma_r \beta x^+ ) &= P_0 \alpha = P_1 \nonumber \\ 
  &\;\, \vdots \\
  P_\infty( \gamma_r \beta x^+ ) &= P_0\alpha \gamma_2 \cdots \gamma_{r-1}  = P_{r-1} \nonumber \\
  P_\infty( \beta x^+ ) &= P_0\alpha \gamma_2 \cdots \gamma_{r-1}\gamma_{r}  = P_{r} \;. \nonumber
\end{align}
Assume now that there exists a partitioning of  $\beta x^+$ with empty beginning (e.g.\  $x^+ = \beta \tilde \alpha \gamma_2 \cdots \gamma_r^\infty$).
The first word used in such a partitioning must be $\beta \tilde \alpha \gamma_2 \cdots \gamma_r$. Replacing this word by the concatenation of the generating words $\alpha \gamma_2 \cdots \gamma_r \beta$, $\tilde \alpha, \gamma_2, \ldots, \gamma_r$ creates a partitioning of $\beta x^+$ with beginning $\alpha \gamma_2 \cdots \gamma_r $, so in this case:
\begin{align*}
P_\infty( \alpha \gamma_2 \cdots \gamma_r \beta x^+ ) &= P_0 \\
P_\infty( \gamma_2 \cdots \gamma_r \beta x^+ ) &= P_0 \cup P_0\alpha = P_0\\
 &\;\, \vdots  \\
 P_\infty(  \gamma_r \beta x^+) &= P_0 \cup P_0\alpha \gamma_2 \cdots \gamma_{r-1}  = P_0 \\ 
P_\infty(  \beta x^+) &= P_0 \cup P_0\alpha \gamma_2 \cdots \gamma_{r-1}\gamma_{r}  = P_0\;.
\end{align*}
The argument above proves that there are no right--rays such that every partitioning of $\beta x^+$ has empty beginning.

\begin{figure}[tb]
\begin{center}
\begin{tikzpicture}
  [bend angle=5,
   clearRound/.style = {circle, inner sep = 0pt, minimum size = 17mm},
   clear/.style = {rectangle, minimum width = 10 mm, minimum height = 6 mm, inner sep = 0pt},  
   greyRound/.style = {circle, draw, minimum size = 1 mm, inner sep =
      0pt, fill=black!10},
  grey/.style = {rectangle, draw, minimum size = 6 mm, inner sep =
    1pt, fill=black!10},
  white/.style = {rectangle, draw, minimum size = 6 mm, inner sep =
    1pt},
   to/.style = {->, shorten <= 1 pt, shorten >= 1 pt, >=stealth', semithick},
   scale=0.9]
  
 \node[grey] (P0) at (0,0) {$P_0$};

 \node[white] (P1) at (-6,-4) {$P_1$};
 \node[white] (P2) at (-3,-4) {$P_2$};
 \node[clear] (dots1) at (0,-4) {$\cdots$};
  \node[white] (Pr-1) at (3,-4) {$P_{r-1}$};
 \node[white] (Pr) at (6,-4) {$P_r$};

 \node[grey] (Pr+1) at (-4,-2) {$P_{r+1}$};
 \node[grey] (Pr+2) at (-1,-2) {$P_{r+2}$};
 \node[clear] (dots2) at (1,-2) {$\cdots$};
 \node[grey] (P2r) at (4,-2) {$P_{2r}$};

  \draw[to, loop above] (P0) to node[auto] {$x$} (P0);

  \draw[to, bend right=45] (P0) to node[auto,swap] {$\alpha$} (P1);
  \draw[to] (P1) to node[auto,swap] {$\gamma_2$} (P2);
  \draw[to] (P2) to node[auto,swap] {$\gamma_3$} (dots1);
  \draw[to] (dots1) to node[auto,swap] {$\gamma_{r-1}$} (Pr-1);
  \draw[to] (Pr-1) to node[auto,swap] {$\gamma_{r}$} (Pr);  
  \draw[to,bend right=45] (Pr) to node[auto,swap] {$\beta$} (P0);
 
  \draw[to, bend left=10] ($(Pr)+(-0.4,0.2)$) to node[near start, above] {$\beta$} (Pr+2);
  \draw[to] (Pr) to node[auto,swap] {$\beta$} (P2r);
  
   \draw[to] (P0) to node[auto,swap] {$\beta,x$} (Pr+1);   
  \draw[to] (Pr+1) to node[auto] {$\tilde \alpha$} (Pr+2);
  \draw[to]  (Pr+2) to node[auto] {$\gamma_{2}$} (dots2); 
  \draw[to] (dots2) to node[auto] {$\gamma_{r-1}$} (P2r); 
  \draw[to,bend right] (P2r) to node[auto,swap] {$\gamma_{r}$} (P0);
  \draw[to,loop above] (P2r) to node[auto] {$\gamma_{r}$} (P2r); 
  \draw[to,bend left=30] (P2r) to node[auto] {$\gamma_{r}$} (Pr+1);
  \draw[to,bend left=10] (P2r) to node[auto] {$\gamma_{r}$} (Pr+2);   

  \draw[to] (P0) to node[auto, swap, near end] {$x$} (Pr+2);
  \draw[to,bend right] (P0) to node[auto,swap] {$x$} (P2r);
    
\end{tikzpicture}
\end{center}
\caption[Building blocks for achieving entropies.]{Left Fischer cover of $\mathsf{X}(L)$ for $L$ defined in (\ref{RJ_eq_R_L}).
An edge labelled $x$ from a vertex $P$ to a vertex $Q$ represents a collection of edges from $P$ to $Q$ such that $Q$ receives an edge with each label from the set 
$\bigcup_{2\leq j \leq r}  \{ \gamma_j  \} \cup \{  \alpha, \tilde \alpha \}$, i.e.\ the collection fills the gaps left by the edges which are labelled explicitly. The border points are coloured grey.} 
\label{RJ_fig_entropy_lfc}
\vspace{1cm}
\end{figure}
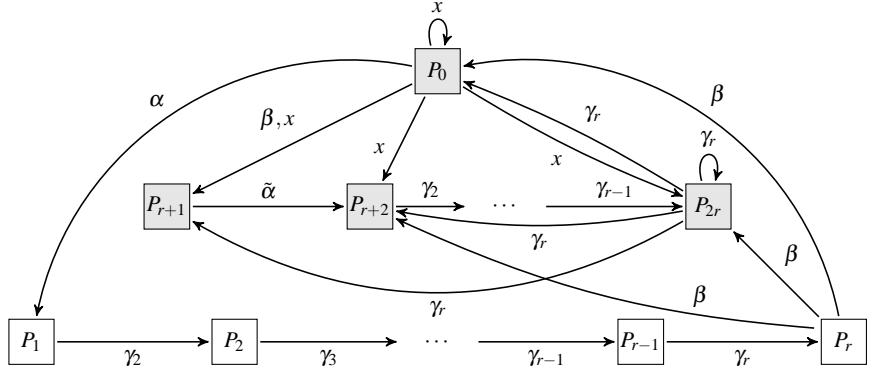

It only remains to investigate right--rays that  have a suffix of $\beta \tilde \alpha \gamma_2 \cdots \gamma_r$ as a prefix.
A partitioning of a right--ray $\gamma_r x^+$ may have empty beginning (e.g.\  $x^+ = \gamma_r^\infty$), beginning $\alpha \gamma_2 \cdots \gamma_{r-1}$ (e.g.\ $x^+ =  \beta \beta \tilde \alpha \gamma_2 \cdots \gamma_r \cdots$ or $x^+ = \beta \tilde \alpha \gamma_2 \cdots \gamma_r^\infty$), or beginning $\beta \tilde \alpha \gamma_2 \cdots \gamma_{r-1}$ (e.g.\  $x^+ = \gamma_r^\infty)$. Note that there is a partitioning with empty beginning if and only if there is a partitioning with beginning $\beta \tilde \alpha \gamma_2 \cdots \gamma_{r-1}$.  
If there exists a partitioning of $\gamma_r x^+$ with beginning $\alpha \gamma_2 \cdots \gamma_{r-1}$, then $\beta$ must be a prefix of $x^+$, so the right--ray $\gamma_r x^+$ has already been considered above.
Hence, it suffices to consider the case where there exists a partitioning of $\gamma_r x^+$ with empty beginning and a partitioning with beginning $\beta \tilde \alpha \gamma_2 \cdots \gamma_{r-1}$ but no partitioning with beginning $\alpha \gamma_2 \cdots \gamma_{r-1}$. In this case, the predecessor sets are 
\begin{align*}
  P_\infty( \gamma_{r} x^+ ) &= P_0 \cup P_0\beta \tilde \alpha \gamma_2 \cdots \gamma_{r-1} = P_{2r}\\
&\;\,\vdots \\
P_\infty( \gamma_2 \cdots \gamma_{r} x^+ ) &= P_0 \cup P_0\beta \tilde \alpha = P_{r+2}\\
 P_\infty( \tilde \alpha \gamma_2 \cdots \gamma_{r} x^+ ) &= P_0 \cup P_0\beta = P_{r+1}\\
P_\infty( \beta \tilde \alpha \gamma_2 \cdots \gamma_{r} x^+ ) &= P_0 
  \cup P_0\alpha \gamma_2 \cdots \gamma_{r} = P_0 \; .
\end{align*}

Now all right--rays have been investigated, so there are exactly $2r +1$ vertices in the left Krieger cover of $\mathsf{X}(L)$. The vertex $P_0$ is the universal border point, and the vertices $P_{r+1}, \ldots, P_{2r}$ are border points, while none of the vertices $P_1, \ldots, P_r$ are border points. This gives the information needed to draw the left Fischer cover.

In \cite{RJ_hong_shin} it is proved that all renewal systems in the class $B$ are SFTs. That proof will also work for the related class $R$ considered here, but the result also follows easily from the structure of the left Fischer cover constructed above \cite[Lemma 5.46]{RJ_johansen_thesis}.
\end{proof}

\begin{lemma}
\label{RJ_lem_entropy_cyclic}
Let $L \in R$ and let $X_f$ be a renewal system obtained from $\mathsf{X}(L)$ by fragmentation. Then the Bowen--Franks group of $X_f$ is cyclic, and the determinant is given by (\ref{RJ_eq_det_R}).
\end{lemma}

\begin{proof}
Let $L \in R$ be defined by (\ref{RJ_eq_R_L}). The symbolic adjacency matrix of the left Fischer cover of $\mathsf{X}(L)$ (shown in Fig.\ \ref{RJ_fig_entropy_lfc}) is 
\begin{displaymath}
A = \left( \begin{array}{c | c c c c c  | c c c c c c}
  \gamma & \alpha & 0 & \cdots  & 0 & 0 & \gamma+\beta & \tilde \alpha' & \gamma'_2 & \cdots  & \gamma'_{r-2} & \gamma'_{r-1} \\
  \hline
  0         & 0         & \gamma_2 & \cdots & 0 & 0    & & & & & &\\
  0         & 0         & 0     &            & 0 & 0    & & & & & &\\
  \vdots & \vdots &        & \ddots &    & \vdots      & & & & 0 & &\\ 
  0         & 0         & 0     &            & 0 &  \gamma_r & & & & & &\\ 
  \beta         & 0         & 0     & \cdots & 0 & 0    & 0  & \beta & \beta & \cdots & \beta & \beta \\   
  \hline
  0         & & &    & & & 0 &  \tilde \alpha & 0     & \cdots & 0 & 0 \\
  0         & & &    & & & 0 & 0     &  \gamma_2 &            & 0 & 0 \\
  0         & & &    & & & 0 & 0     & 0     &            & 0 & 0 \\
  \vdots & & & 0 & & & \vdots   &        &        & \ddots &    & \vdots   \\
  0         & & &    & & & 0 & 0     & 0     &            & 0 &  \gamma_{r-1} \\
   \gamma_r      & & &    & & &  \gamma_r &  \gamma_r     &  \gamma_r     & \cdots &  \gamma_r &  \gamma_r \\   
 \end{array} \right) \; ,
 \end{displaymath} 
where $ \gamma = \alpha+\tilde \alpha+ \sum_{k=2}^{r-1}  \gamma_k$, $\tilde \alpha' =  \gamma -  \tilde \alpha$,  and $\gamma'_k =  \gamma -  \gamma_k$. Index the rows and columns of $A$ by $0,\ldots, 2r$ in correspondence with the names used for the vertices above, and note that the column sums of the columns $0, r+1, \ldots, 2r$ are all equal to $\alpha+\tilde\alpha +\beta+ \sum_{k=2}^{r}  \gamma_k$.

If $X_f$ is a fragmentation of $\mathsf{X}(L)$, then the (non--symbolic) adjacency matrix $A_f$ of the underlying graph of the left Fischer cover of $X_f$ is obtained from $A$ by replacing $\alpha, \tilde \alpha, \beta, \gamma_2, \ldots , \gamma_r$ by positive integers (see Remark \ref{RJ_rem_fragmentation}).
To put $\RJId - A_f$ into Smith normal form, begin by adding each row from number $r+1$ to $2r-1$ to the first row, and subtract the first column from column $r+1, \ldots, 2r$ to obtain
\begin{displaymath}
\RJId - A_f \rightsquigarrow 
\left( \begin{array}{c | c c c c c  | c c  c c c}
  1- \gamma & - \alpha & 0 & \cdots  & 0 & 0 & -\beta & 0 & \cdots  & 0 & -1 \\
  \hline
  0         & 1         & -\gamma_2 & \cdots & 0 & 0    & & & & &\\
  0         & 0         & 1     &              & 0 & 0    & & & & &\\
  \vdots & \vdots &        & \ddots &    & \vdots      & &  & 0 & &\\ 
  0         & 0         & 0     &            & 1 & - \gamma_r &  & & & &\\ 
  -\beta         & 0         & 0     & \cdots & 0 & 1    & \beta  & 0 & \cdots & 0 & 0\\   
  \hline
  0         & & &    & & & 1 & -\tilde \alpha & \cdots & 0 & 0 \\
  0         & & &    & & & 0 & 1     &     & 0 & 0 \\
  \vdots & & & 0 & & & \vdots   &         & \ddots &    & \vdots   \\
  0         & & &    & & & 0 & 0         &            & 1 & - \gamma_{r-1} \\
  - \gamma_r      & & &    & & & 0 & 0         & \cdots & 0 &  1 \\   
\end{array} \right) \: .
\end{displaymath} 
Using row and column addition, this matrix can be further reduced to 
\begin{displaymath}
\rightsquigarrow 
\left( \begin{array}{c | c c c   |   c c c}
  1- \gamma - b & 0 &  \cdots   & 0 & 0 & \cdots  &  t \\
  \hline
  0         & 1         &   \cdots  & 0             &  & &\\
  \vdots & \vdots &   \ddots  & \vdots      & &  0 &\\ 
  0         & 0         &   \cdots  & 1             &   &   & \\   
  \hline
  0         & &    &    & 1          & \cdots   & 0 \\
  \vdots & & 0 &    & \vdots  &  \ddots   & \vdots   \\
  -\gamma_r         & &    &    &  0         &  \cdots   &  1 \\   
 \end{array} \right)
 \begin{array}{l}
\\
 b = \alpha \beta \gamma_2 \cdots \gamma_r \\
 \\
 t = \tilde \alpha \gamma_2 \cdots \gamma_{r-1} (b - \beta) -1 \; . \\
 \\
 \end{array}
\end{displaymath}
Hence, the Bowen--Franks group of $X_f$ is cyclic, and the determinant is 
\begin{equation} \label{RJ_eq_det_R}
\det(\RJId - A) 
                   = 1 - \alpha - \tilde \alpha - \sum_{k=2}^r \gamma_k - (\alpha+\tilde \alpha) \beta \gamma_2 \cdots \gamma_r  + \alpha \tilde \alpha \beta(\gamma_2 \cdots \gamma_r)^2 \;.   
\end{equation}
\end{proof}

\begin{theorem}
For each $L \in H$, the renewal system $\mathsf{X}(L)$ has cyclic Bowen--Franks group and determinant given by (\ref{RJ_eq_entropy_det}). \label{RJ_thm_H_classification}
\end{theorem}

\begin{proof}
By Lemma \ref{RJ_lem_entropy_fe}, there exist $L_1, \ldots , L_m \in R$, $L_0 = \{a \}$ for some letter $a$ that does not appear in any of the lists,  and a fragmentation $Y_f$ of $Y = \mathsf{X}( \bigcup_{j=0}^m L_j )$ such that $Y_f \sim_{\textrm{FE}} \mathsf{X}(L)$. For $1 \leq j \leq m$, let 
$  L_j = \{ \alpha_j, \tilde \alpha_j, \gamma_{j,k}, 
 \alpha_j \gamma_{j,2} \cdots \gamma_{j,r_j} \beta_j, \beta_j \tilde \alpha_j \gamma_{j,2} \cdots \gamma_{j,r_j} \mid 2 \leq k \leq r_j\} , r_j \in \mathbb{N}$.
Each $L_j$ is left--modular by Lemma \ref{RJ_lem_entropy_lfc},
so $Y$ is an SFT, and the left Fischer cover of $Y$ can be constructed using the technique from Sec.\ \ref{RJ_sec_rs_lfc}: Identify the universal border points in the left Fischer covers of $\mathsf{X}(L_0), \ldots , \mathsf{X}(L_m)$, and draw additional edges to the border points corresponding to the edges terminating at the universal border points in the individual left Fischer covers. Hence, the symbolic adjacency matrix $A$ of the left Fischer cover of $Y$ is 
\begin{displaymath}
A = 
\left( \!\!\! \begin{array}{c | c | c c c c  | c c c c | c | c }
  \gamma &             & \alpha_j & 0 & \cdots  & 0 &  \gamma+\beta_ j & \tilde \alpha'_{ j} & \cdots   & \gamma'_{ j,{r_ j-1}}  & \cdots & \gamma'_{i,k}\\
   \hline
                &\ddots    &             & & & & & & & & &  \\
   \hline
  0            &              & 0           & \gamma_{ j,2} & \cdots  & 0    & & & & & &\\
  0            &              & 0           & 0     &                & 0    & & & & & & \\
  \vdots    &              & \vdots   &        & \ddots      & \vdots     & &  & 0 & & &\\ 
  0            &              & 0           & 0     &                & \gamma_{ j,r} & & & &  & & \\ 
  \beta_ j  &              & 0           & 0     & \cdots  & 0    & 0  & \beta_ j & \cdots & \beta_ j & & \beta_ j \\   
  \hline
  0           &               &             &    & & & 0 & \tilde \alpha_j &  \cdots & 0 &  &\\
  0           &               &             &    & & & 0 & 0     &             & 0 & & \\
  \vdots   &               &             & & 0 & & \vdots   &        &      \ddots &   \vdots   & & \\
  0           &               &             &    & & & 0 & 0     &             &  \gamma_{ j,{r_ j-1}}  & &\\
  \gamma_{ j,r_ j}  & &             &    & & & \gamma_{ j,r_ j} &  \gamma_{ j,r_ j}  & \cdots &  \gamma_{ j,r_ j}  & & \gamma_{ j,r_ j}\\  
  \hline
   &                           & &    & & &  &     &  & &  \ddots & \\     
  \hline 
   &                           & &    & & &  &     &  & &   & \ddots      
\end{array} \!\!\! \right) \; .
\end{displaymath} 
where $1 \leq j \leq m$, $\gamma = a + \sum_{j=1}^m \left( \alpha_j+\tilde\alpha_j + \sum_{k=2}^{r_j-1} \gamma_{j,k} \right)$, $\tilde \alpha_j' = \gamma - \tilde \alpha_j$, and $\gamma_{j,k}' = \gamma - \gamma_{j,k}$. This matrix has blocks of the same form as in the $m=1$ case considered in Lemma \ref{RJ_lem_entropy_lfc}. The $j$th block is shown together with the first row and column of the matrix---which contain the connections between the $j$th block and the universal border point $P_0$---and together with an extra column representing an arbitrary border point in a different block. Such a border point in another block will receive edges from the $j$th block with the same sources and labels as the edges that start in the $j$th block and terminate at the universal border point $P_0$. 

Let $Y_f$ be a fragmentation of $Y$. Then the (non--symbolic) adjacency matrix $A_f$ of the underlying graph of the left Fischer cover of $Y_f$ is obtained by replacing the entries of $A$ by positive integers as described in Remark \ref{RJ_rem_fragmentation}.
In order to put $\RJId - A_f$ into Smith normal form, first add rows $r_j+1$ to $2r_j-1$
in the $j$th block to the first row for each $j$, and then subtract the first column from every column corresponding to a border point in any block.
In this way, $\RJId-A_f$ is transformed into:
\begin{displaymath}
\left(\!\!\! \begin{array}{c | c | c c c c  | c c c c | c  }
  1-\gamma &    & -\alpha_j & 0 & \cdots  & 0 &  -\beta_j & 0 & \cdots   & -1  & \\
  \hline
   &\!\ddots\!       &         &  &  &   & & & & &  \\
   \hline
  0  &       & 1         & -\gamma_{j,2} & \cdots  & 0    & & & & & \\
  0   &      & 0         & 1     &                & 0    & & & & &  \\
  \vdots& & \vdots &        & \ddots      & \vdots     & &  & 0 & & \\ 
  0  &       & 0         & 0     &                & -\gamma_{j,r_j} & & & &  &  \\ 
  -\beta_j  &       & 0         & 0     & \cdots  & 1    & \beta_j  & 0 & \cdots & 0 &  \\   
  \hline
  0   &    & &    & & & 1 & -\tilde \alpha_j &  \cdots & 0 &  \\
  0    &   & &    & & & 0 & 1     &             & 0 &  \\
  \vdots & & & & 0 & & \vdots   &        &      \ddots &   \vdots   &  \\
  0   &    & &    & & & 0 & 0     &             &  -\gamma_{j,{r_j-1}}  & \\
  -\gamma_{j,r_j}  & & &    & & & 0 & 0 & \cdots &  1  & \\  
  \hline
   &     & &    & & &  &     &  & & \! \ddots  \\     
 \end{array} \!\!\! \right) \; .
 \end{displaymath} 
By using row and column addition, and by disregarding rows and columns where the only non--zero entry is a diagonal $1$, $\RJId - A$ can be further reduced to 
\begin{displaymath}
\left( \begin{array}{c c c c c c}
S                           & t_1  & t_2  & \cdots & t_m          \\
-\gamma_{1,r_1}   & 1     & 0     &            &  0               \\
-\gamma_{2,r_2}   & 0     & 1     &            &  0               \\
 \vdots                   &        &        & \ddots &  \vdots       \\
-\gamma_{m,r_m} & 0     & 0     & \cdots &  1 \\
\end{array}\right)
\quad \begin{array}{l}
b_j = \alpha_j \beta_j \gamma_{j,2} \cdots \gamma_{j,r_j} \\ \\
t_j = \tilde \alpha_j \gamma_{j,2} \cdots \gamma_{j,r-1} (b_j-\beta_j) -1 \\ \\
S = 1 - \gamma - \sum_{j=1}^m b_j
\end{array}
\; .
\end{displaymath}
Hence, the Bowen--Franks group is cyclic and the determinant is 
\begin{equation}
\label{RJ_eq_entropy_det}
\det( \RJId - A_f ) = 1 - \gamma + \sum_{j=1}^m \left(  \gamma_{j,r_j} t_j  - b_j \right)\; . 
\end{equation}
\end{proof}

\noindent With the results of \cite{RJ_hong_shin}, this gives the following result.

\begin{corollary}
When $\log \lambda$ is the entropy of an SFT, there exists an SFT renewal system $X(L)$  with cyclic Bowen--Franks group such that $h(\mathsf{X}(L)) = \log \lambda$.
\end{corollary}

\section{Towards the range of the Bowen--Franks invariant}
\label{RJ_sec_rs_range}
In the following, it will be proved that the range of the Bowen--Franks invariant over the class of SFT renewal systems contains a large class of pairs of signs and finitely generated abelian groups. 
First, the following special case will be used to show that every integer is the determinant of an SFT renewal system.

\index{renewal system!with positive determinant}
\begin{example}
\label{RJ_ex_pos_det}
Consider the generating list \fxnote{Corrected the generating list}
\begin{equation}\label{RJ_eq_pos_det}
   L = \{ a, \alpha, \tilde \alpha, \gamma, \alpha \gamma \beta, \beta \tilde \alpha \gamma \} \; .
\end{equation}
By Lemma \ref{RJ_lem_entropy_lfc}, $L$ is left--modular, $\mathsf{X}(L)$ is an SFT, and the symbolic adjacency matrix of the left Fischer cover of $\mathsf{X}(L)$ is  
\begin{equation}
\label{RJ_eq_pos_det_adj}
A = \left( \begin{array}{c | c c | c c}
       a +\alpha+\tilde \alpha & \alpha    & 0 & a+\alpha +\tilde \alpha+\beta & a+\alpha \\
       \hline
       0                   &   0      & \gamma &  0                         &   0          \\
       \beta             &   0      & 0            &  0                         &   \beta     \\
\hline
       0                   &   0      & 0            &  0                         &   \tilde \alpha        \\
       \gamma        &   0      & 0            &  \gamma              &  \gamma                                      
\end{array} \right) \; .
\end{equation}
\noindent
By fragmenting $\mathsf{X}(L)$, it is possible to construct an SFT renewal system for which the (non--symbolic) adjacency matrix of the underlying graph of the left Fischer cover has this form with $a, \alpha, \tilde \alpha, \beta, \gamma \in \mathbb{N}$ as described in Remark \ref{RJ_rem_fragmentation}. Let $A_f$ be such a matrix. This is a special case of the shift spaces considered in Theorem \ref{RJ_thm_H_classification}, so the Bowen--Franks group is cyclic and the determinant is 
$
   \det(\RJId - A_f) =  \beta \alpha \tilde \alpha \gamma^2 - \alpha \beta  \gamma - \tilde \alpha \beta  \gamma -\alpha - \tilde \alpha-\gamma-a +1
$.
\end{example} 


\begin{theorem}\label{RJ_thm_rs_det_range}
Any $k \in \mathbb{Z}$ is the determinant of an SFT renewal system with cyclic Bowen--Franks group.
\end{theorem}

\begin{proof}
Consider the renewal system from Example \ref{RJ_ex_pos_det} in the case $\alpha = \tilde \alpha = \beta = 1$, where the determinant is
$   \det(\RJId - A_f)   =  \gamma^2 - 3\gamma - a -1$,
and note that the range of this polynomial is $\mathbb{Z}$.
\end{proof}

\index{renewal system!with non--cyclic group}
\index{X@$\RJXd{n_1,\dots,n_k}$}
All renewal systems considered until now have had cyclic Bowen--Franks groups, so the next goal is to construct a class of renewal systems exhibiting non--cyclic groups. 
Let $k \geq 2$, $\mathcal{A} = \{ a_1, \ldots , a_k \}$, and let $n_1, \ldots , n_k \geq 2$ with $\max_{i}\{n_{i}\} > 2$.
The goal is to define a generating list, $L$, for which $\mathsf X(L) = \mathsf X_\mathcal F$ with $\mathcal{F} = \{ a_i^{n_{i}}  \}$. For each $1 \leq i \leq k$, define
\begin{multline}
\label{RJ_eq_Ld}
  L_i = \{  a_j a_i^l \mid j \neq i \textrm{ and } 0 < l < n_i -1\}  \\ \cup 
            \{  a_m a_j a_i^l \mid m \neq j \neq i \textrm{ and } 0 < l < n_i -1\} \; .
\end{multline}
Define $L = \bigcup_{i=1}^k L_i \neq \emptyset$, and $\RJXd{n_1,\dots,n_k} = \mathsf{X}(L)$.

\begin{lemma}
\label{RJ_lem_Xd}
Define the renewal system $\RJXd{n_1,\dots,n_k}$ as above. Then $\RJXd{n_1,\dots,n_k} = \mathsf X_\mathcal F$ with $\mathcal F = \{ a_i^{n_{i}}  \}$, so $\RJXd{n_1,\dots,n_k}$ is an SFT. The symbolic adjacency matrix of the left Fischer cover of $\RJXd{n_1,\dots,n_k}$ is the matrix in (\ref{RJ_eq_diag_A}). 
\end{lemma}

\begin{proof}
Note that for each $i$, $a_i^{n_{i}} \notin \mathcal B (\RJXd{n_1,\dots,n_k})$ by construction. For $1 < l < n_i-1$ and $j \neq i$ the word $a_j a_i^{l}$ has a partitioning in $\RJXd{n_1,\dots,n_k}$ with empty beginning and end. Hence, $a_{i_1} a_{i_2}^{l_2} a_{i_3}^{l_3} \cdots a_{i_m}^{l_m}$ has a partitioning with empty beginning and end whenever $i_j \neq i_{j+1}$, $1 < l_j < n_{i_j}$ for all $1 < j < m$, and $0 < l_m < n_{i_m} -1$. Given $i_1, \ldots , i_m \in \{1, \ldots, k\}$ with $i_j \neq i_{j+1}$ and $m \geq 2$, the word $a_{i_1} a_{i_2} \cdots a_{i_m}$ has a partitioning with empty beginning and end. Hence, every word that does not contain one of the words $ a_i^{n_{i}}$ has a partitioning, so $\RJXd{n_1,\dots,n_k} = \mathsf X_\mathcal F$ for $\mathcal F = \{ a_i^{n_{i}}\}$. 

\begin{figure}[b]
\vspace{0.8cm}
\begin{center}
\begin{tikzpicture}
 [bend angle=10,
   clearRound/.style = {circle, inner sep = 0pt, minimum size = 17mm},
   clear/.style = {rectangle, minimum width = 17 mm, minimum height = 6 mm, inner sep = 0pt},  
   greyRound/.style = {circle, draw, minimum size = 1 mm, inner sep =
      0pt, fill=black!10},
   grey/.style = {rectangle, draw, minimum size = 6 mm, minimum height = 8mm, inner sep =
      1pt, fill=black!10},
   white/.style = {rectangle, draw, minimum size = 6 mm, minimum height = 8mm, inner sep =
      1pt},
   to/.style = {->, shorten <= 1 pt, >=stealth', semithick}]
  
  \node[white] (an1) at (0,2) {$P_\infty \left(a_i^{n_i-1} \right)$};
  \node[white] (an2) at (3,2) {$P_\infty \left(a_i^{n_i-2}\right)$};
  \node[clear] (dots) at (6,2) {$\cdots$};  
  \node[grey] (a) at (9,2) {$P_\infty \left(a_i\right)$}; 

  \node[white] (bn1) at (0, 0) {$P_\infty \left(a_j^{n_j-1}\right)$};
  \node[white] (bn2) at (3,0) {$P_\infty \left(a_j^{n_j-2}\right)$};
  \node[clear] (dots2) at (6,0) {$\cdots$};  
  \node[grey] (b) at (9,0) {$P_\infty \left(a_j \right)$};

  \draw[to] (an1) to node[auto] {$a_i$} (an2);
  \draw[to] (an2) to node[auto] {$a_i$} (dots);
  \draw[to] (dots) to node[auto] {$a_i$} (a);  

  \draw[to] (bn1) to node[auto,swap] {$a_j$} (bn2);
  \draw[to] (bn2) to node[auto,swap] {$a_j$} (dots2);
  \draw[to] (dots2) to node[auto,swap] {$a_j$} (b);  

  \draw[to, bend left] (a) to node[auto] {$a_i$} (b);
  \draw[to] (a) to node[auto, near end] {$a_i$} (bn2);
  \draw[to] (a) to node[auto,swap, very near end] {$a_i$} (bn1);  

  \draw[to, bend left] (b) to node[auto] {$a_j$} (a);
  \draw[to] (b) to node[auto,swap, near end] {$a_j$} (an2);
  \draw[to] (b) to node[auto,very near end] {$a_j$} (an1);  

\end{tikzpicture}
\end{center}
\caption[Building non--cyclic Bowen--Franks groups.]{Part of the left Fischer cover of $\RJXd{n_1,\dots,n_k}$.
The entire graph can be found by varying $i$ and $j$.
The border points are  coloured grey.} 
\label{RJ_fig_Xd_LFC}
\end{figure}
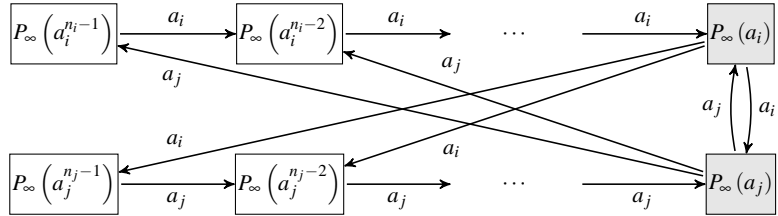

To find the left Fischer cover of  $\RJXd{n_1,\dots,n_k}$, it is first necessary to determine the predecessor sets. Given $1 \leq i \leq k$ and $j \neq i$ 
\begin{align}
	P_\infty(a_i a_j \cdots) &= \{ x^- \in \RJXd{n_1,\dots,n_k}^-  |  x_{-n_i +1} \cdots x_0  \neq a_i^{n_i-1} \} \nonumber \\
	P_\infty(a_i^2 a_j \cdots) &= \{ x^- \in \RJXd{n_1,\dots,n_k}^-  |  x_{-n_i +2} \cdots x_0  \neq a_i^{n_i-2} \} \label{RJ_eq_rs_diag}\\
	&\vdots  \nonumber \\
	P_\infty(a_i^{n_i-1} a_j \cdots) &= \{ x^- \in \RJXd{n_1,\dots,n_k}^-  |  x_0  \neq a_i \} \nonumber .
\end{align}
Only the first of these predecessor sets is a border point. 
Equation \ref{RJ_eq_rs_diag} gives all the information necessary to draw the left Fischer cover of $\RJXd{n_1,\dots,n_k}$. A part of the left Fischer cover is shown in Fig.\ \ref{RJ_fig_Xd_LFC},
and the corresponding symbolic adjacency matrix is: \vspace{-0.5 em}\begin{align}
 	\nonumber
	&\hspace{14 pt}
	\begin{array}{c c c c}
       \overbrace{
       	\phantom{\begin{matrix}
			a_1 &  \cdots & a_1 & a_1
		\end{matrix}}
	}^{n_1-1}		
	&  
  	\overbrace{
       	\phantom{\begin{matrix}
			a_2 &  \cdots & a_2 & a_2
		\end{matrix}}
	}^{n_2-1}
	&
	\phantom{\cdots}
	&		
	\overbrace{
       	\phantom{\begin{matrix}
			a_k &  \cdots & a_k & a_k
		\end{matrix}}
	}^{n_k-1}
	\end{array} \\[-1.5em]
       &\left( \begin{array}{ c  | c |  c |  c } 
	\begin{matrix}
	0 &   \cdots & 0 & 0 \\
	a_1 &  \cdots & 0 & 0 \\
	   &     \ddots &    &    \\
	0 &   \cdots & a_1 & 0		 
	\end{matrix}		
	&
	\begin{matrix}
	a_1 &   \cdots & a_1 & a_1\\
	0 &   \cdots & 0 & 0	\\
	   &   \ddots &    &    \\
	0 &  \cdots & 0 & 0	
	\end{matrix}		
	&
	\cdots
	&
	\begin{matrix}
	a_1 &  \cdots & a_1 &  a_1\\
	0 &  \cdots & 0 & 0	\\
	   &  \ddots &    &    \\
	0 &  \cdots & 0 & 0	
	\end{matrix}		\\
	\hline
	\begin{matrix}
	a_2 &  \cdots & a_2 &  a_2\\
	0 &   \cdots & 0 & 0	\\
	   &   \ddots &    &    \\
	0 &  \cdots & 0 & 0	
	\end{matrix}		
	&
	\begin{matrix}
	0 &  \cdots & 0 & 0 \\
	a_2 &  \cdots & 0 & 0 \\
	   &   \ddots &    &    \\
	0 &  \cdots & a_2 & 0		 
	\end{matrix}		
	&
	\cdots
	&
	\begin{matrix}
	a_2 & \cdots & a_2 &  a_2\\
	0 & \cdots & 0 & 0	\\
	   &  \ddots &    &    \\
	0 &  \cdots & 0 & 0	
	\end{matrix}		\\
	\hline
	\vdots & \vdots & \ddots & \vdots \\
	\hline
	\begin{matrix}
	a_k &   \cdots & a_k &  a_k\\
	0 &  \cdots & 0 & 0	\\
	   &   \ddots &    &    \\
	0 &  \cdots & 0 & 0	
	\end{matrix}		
	&
	\begin{matrix}
	a_k &  \cdots & a_k &  a_k\\
	0 &  \cdots & 0 & 0	\\
	   &  \ddots &    &    \\
	0 & \cdots & 0 & 0	
	\end{matrix}		
	&
	\cdots
	&
	\begin{matrix}
	0 &  \cdots & 0 & 0 \\
	a_k &  \cdots & 0 & 0 \\
	   &     \ddots &    &    \\
	0 &   \cdots & a_k & 0		 
	\end{matrix}		\\
\end{array} \right) \; . \label{RJ_eq_diag_A} 
\end{align} \end{proof}

Let $A$ be the (non--symbolic) adjacency matrix of the underlying graph of the left Fischer cover of $\RJXd{n_1,\dots,n_k}$ constructed above.
Then it is possible to do the following transformation by row and column addition  
\begin{align*}
	\RJId - A  &\rightsquigarrow 
	\begin{pmatrix}
	1        & 1-n_2 & 1-n_3 & \cdots  & 1-n_k  \\
	1-n_1 & 1        & 1-n_3 & \cdots  & 1-n_k  \\
	1-n_1 & 1-n_2 & 1        & \cdots  & 1-n_k  \\
       \vdots & \vdots & \vdots & \ddots & \vdots     \\
	1-n_1 & 1-n_2 & 1-n_3 & \cdots  & 1 		 
	\end{pmatrix} 
	\rightsquigarrow
	\begin{pmatrix}
	x      & 1       & 1       & \cdots  & 1  \\
	-n_1 & n_2  & 0       & \cdots  & 0  \\
	-n_1 & 0       & n_3  & \cdots  & 0  \\
       \vdots & \vdots & \vdots & \ddots & \vdots     \\
	-n_1 & 0       & 0 & \cdots  & n_k 		 
	\end{pmatrix} \; ,
\end{align*}
where $x = 1-(k-1) n_1$. The determinant of this matrix is
\begin{displaymath}
\det(\RJId - A) = n_2 \cdots n_k \left( x + \sum_{i=2}^k \frac{n_1}{n_i} \right) 
                   = - n_1 n_2 \cdots n_k \left( k-1 - \sum_{i=1}^k \frac{1}{n_i} \right)
                   < 0 \; .
\end{displaymath}
The inequality is strict since $k-1 - \sum_{i=1}^k \frac{1}{n_i} > \frac{k}{2} -1 \geq 0$.
Given concrete $n_1, \ldots, n_k$, it is straightforward to compute the Bowen--Franks group of $\RJXd{n_1, \ldots, n_k}$, but it has not been possible to derive a general closed form for this group.

\begin{proposition}
\label{RJ_prop_non-cyclic_BF}
Let $n_1, \ldots, n_k \geq 2$ with $n_i | n_{i-1}$ for $2 \leq i \leq k$ and $n_1 > 2$. Let $m = n_1 n_2 ( k-1 - \sum_{i=1}^k \frac{1}{n_i} )$, then $\RJBF_+(\RJXd{n_1, \ldots, n_k}) = - \mathbb{Z} /m\mathbb{Z} \oplus \mathbb{Z} /n_3\mathbb{Z} \oplus \cdots \oplus \mathbb{Z} / n_k\mathbb{Z}$.
\end{proposition}

\begin{proof}
By the arguments above, $\RJXd{n_1, \ldots , n_k}$ is conjugate to an edge shift with adjacency matrix $A$ such that the following transformation can be carried out by row and column addition
\begin{displaymath}
	\RJId - A \rightsquigarrow 
	\begin{pmatrix}
	y   & 1       & 1       & \cdots  & 1  \\
	0 & n_2  & 0       & \cdots  & 0  \\
	 0   & 0   & n_3  & \cdots  & 0  \\
       \vdots & \vdots & \vdots & \ddots & \vdots     \\
	 0   & 0       & 0 & \cdots  & n_k 		 
	\end{pmatrix}
	\rightsquigarrow
	\begin{pmatrix}
	 0   & 1       & 0   & \cdots  & 0  \\
	m &  0  & 0       & \cdots  & 0  \\
	 0   & 0   & n_3  & \cdots  & 0  \\
       \vdots & \vdots & \vdots & \ddots & \vdots     \\
	 0   & 0       & 0 & \cdots  & n_k 		 
	\end{pmatrix} \; ,	
\end{displaymath}
where $y = -n_1 \left( k-1- \sum_{i=1}^k 1/n_i \right)$.	
It follows that the Smith normal form of $\RJId - A$ is $\RJdiag(m, n_3, \ldots , n_k)$, and $\det(\RJId - A) < 0$.
\end{proof}

Let  $G$ be a finite direct sum of finite cyclic groups. Then Proposition \ref{RJ_prop_non-cyclic_BF} shows that $G$ is a subgroup of the Bowen--Franks group of some SFT renewal system, but it is still unclear whether $G$ itself is also the Bowen--Franks group of a renewal system since the term $\mathbb{Z} / m\mathbb{Z}$ in the statement of Proposition \ref{RJ_prop_non-cyclic_BF}  is determined by the other terms. Furthermore, the groups constructed in Proposition \ref{RJ_prop_non-cyclic_BF} are all finite. Other techniques can be used to construct renewal systems with groups such as $\mathbb{Z} / (n+1) \oplus \mathbb{Z}$ \cite[Ex. 5.54]{RJ_johansen_thesis}.

\index{renewal system!with positive determinant}
\index{renewal system!with non--cyclic group}
The determinants of all the renewal systems with non--cyclic Bowen--Franks groups considered above were negative or zero, so the next goal is to construct a class of SFT renewal systems with positive determinants and non--cyclic Bowen--Franks groups.

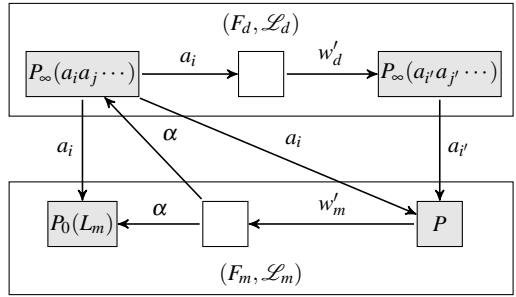
\begin{figure}[b]
\begin{tikzpicture}
  [bend angle=30,
   clearRound/.style = {circle, inner sep = 0pt, minimum size = 17mm},
   clear/.style = {rectangle, minimum width = 5 mm, minimum height = 5 mm, inner sep = 0pt},  
   greyRound/.style = {circle, draw, minimum size = 1 mm, inner sep =
      0pt, fill=black!10},
   grey/.style = {rectangle, draw, minimum size = 6 mm, inner sep =
      1pt, fill=black!10},
    white/.style = {rectangle, draw, minimum size = 6 mm, inner sep =
      1pt},
  whiteBig/.style = {rectangle, draw, minimum height = 1.5 cm, minimum width = 6.7 cm, inner sep =
      1pt},
   to/.style = {->, shorten <= 1 pt, >=stealth', semithick},
   scale=0.95] 

  \node[whiteBig] (Fd) at (0,1.5)   {};
  \node[whiteBig] (Fm) at (0,-1) {};
  \node (space) at (4.5,0) {};
  
  \node[grey] (Pi) at (-2.5,1.3) {$P_\infty(a_ia_j \cdots)$};
  \node[white] (Q') at (0,1.3) {};  
  \node[grey] (Pi') at (2.5,1.3) {$P_\infty(a_{i'}a_{j'} \cdots)$};
  
  \node[grey] (P0) at (-2.5,-0.8) {$P_0(L_m)$};
  \node[grey] (P) at (2.5,-0.8) {$P$};
  \node[white] (Q) at (-0.5,-0.8) {};  
  \node[clear] (Fdtext) at (0,2) {$(F_d, \mathcal{L}_d)$};  
  \node[clear] (Fmtext) at (0,-1.5) {$(F_m, \mathcal{L}_m)$};    
   
  \draw[to] (Pi) to node[auto] {$a_i$} (P);
  \draw[to] (Pi) to node[auto,swap] {$a_i$} (P0);
  \draw[to] (Pi) to node[auto] {$a_i$} (Q');
  \draw[to] (Q') to node[auto] {$w'_d$} (Pi');
  \draw[to] (Pi') to node[auto] {$a_{i'}$} (P);      
  \draw[to] (P) to node[auto,swap] {$w'_m$} (Q);
  \draw[to] (Q) to node[auto,swap] {$\alpha$} (P0);
  \draw[to] (Q) to node[auto,swap] {$\alpha$} (Pi);

\end{tikzpicture}
\hspace{\stretch{1}}
\sidecaption
\caption[Construction of the Fischer cover of a sum.]{Construction of the left Fischer cover considered in Lemma \ref{RJ_lem_addition_Xd}. Here, $w'_m \alpha = w_m \in L_m^*$ and $w'_d a_{i'} = w_d \in \mathcal{B}(\RJXd{n_1, \ldots, n_k})$ with $\RJleftl(w_d) \neq a_i$. Border points are coloured grey.} 
\label{RJ_fig_addition_Xd}
\end{figure}

\begin{lemma}
\label{RJ_lem_addition_Xd}
Let $L_d$ be the generating list of $\RJXd{n_1, \ldots , n_k}$ as defined in (\ref{RJ_eq_Ld}), and let $(F_d, \mathcal{L}_d)$ be the left Fischer cover of $\RJXd{n_1, \ldots , n_k}$.
Let $L_m$ be a left--modular generating list for which $\mathsf{X}(L_m)$ is an SFT with left Fischer cover $(F_m, \mathcal{L}_m)$. 
For $L_{d+m} = L_d \cup L_m \cup_{i=1}^k \{ a_i w \mid w \in L_m \}$,   $\mathsf{X}(L_{d+m})$ is an SFT for which the left Fischer cover is obtained by adding the following connecting edges to the disjoint union of $(F_d, \mathcal{L}_d)$ and $(F_m, \mathcal{L}_m)$ (sketched in Fig. \ref{RJ_fig_addition_Xd}):
\begin{itemize}
\item For each $1 \leq i \leq k$ and each $e \in F_m^0$ with $r(e) = P_0(L_m)$ draw an edge $e_i$ with $s(e_i) = s(e)$ and $r(e_i) = P_\infty(a_i a_j \ldots)$  labelled $\mathcal{L}_m(e)$.
\item For each $1 \leq i \leq k$ and each border point $P \in F_m^0$ draw an edge labelled $a_i$ from $P_\infty(a_i a_j \ldots)$ to  $P$.
\end{itemize}
\end{lemma}

\begin{proof}
 Let $(F_{d+m}, \mathcal{L}_{d+m})$ be the labelled graph defined in the lemma
 and sketched in Fig.\ \ref{RJ_fig_addition_Xd}.
The graph is left--resolving, predecessor--separated, and irreducible by construction, so it is the left Fischer cover of some sofic shift $X$ \cite[Cor. 3.3.19]{RJ_lind_marcus}. The first goal is to prove that $X = \mathsf{X}(L_{d+m})$.
By the arguments used in the proof of Lemma \ref{RJ_lem_Xd}, any word of the form $a_{i_0}w_m a_{i_1} a_{i_2}^{l_i} \ldots a_{i_p}^{l_p}$ where 
$w_m \in L_m^*$,
$p \in \mathbb{N}$, 
$i_j \neq i_{j+1}$ and $l_j < n_{i_j}$ for $1 < j < p$,
and $1 \leq l_p < n_{i_p} -1$
has a partitioning with empty beginning and end in $\mathsf{X}(L_{d+m})$. 
Hence, $\mathcal{B}(\mathsf{X}(L_{d+m}))$ is the set of factors of concatenations of words from
$
\left\{ w_m a_i w_d \mid w_m \in L_m^*, 1 \leq i \leq k, w_d \in \mathcal{B}(\RJXd{n_1, \ldots, n_k}), \RJleftl(w_d) \neq a_i \right \}
$.
Since $L_m$ is left--modular, a path $\lambda \in F_m^*$ with $r(\lambda) = P_0(L_m)$ has $\mathcal{L}_m(\lambda) \in L_m^*$ if and only if $s(\lambda)$ is a border point in $F_m$.
Hence, the language recognised by the left Fischer cover $(F_{d+m}, \mathcal{L}_{d+m})$ is precisely the language of $\mathsf{X}(L_{d+m})$.
 
It remains to show that $(F_{d+m}, \mathcal{L}_{d+m})$ presents an SFT. 
Let $1 \leq i \leq k$ and let $\alpha \in \mathcal{B}(\mathsf{X}(L_m))$, then any labelled path in $(F_{d+m}, \mathcal{L}_{d+m})$ with $a_i \alpha$ as a prefix must start at $P_\infty(a_ia_j\cdots)$.
Similarly, if there is a path $\lambda \in F_{d+m}^*$ with $\alpha a_i$ as a prefix of $\mathcal{L}_{d+m}(\lambda)$, then there must be unique vertex $v$ emitting an edge labelled $\alpha$ to $P_0(L)$, and $s(\lambda) = v$.
Let $x \in \mathsf{X}_{(F_{d+m},\mathcal{L}_{d+m})}$. If there is no upper bound on set of $i \in \mathbb{Z}$ such that $x_i \in \{a_1, \ldots, a_k\}$ and $x_{i+1} \in \mathcal{A}(\mathsf{X}(L_m))$ or vice versa, then the arguments above and the fact that the graph is left--resolving prove that there is only one path in $(F_{d+m},\mathcal{L}_{d+m})$ labelled $x$. If there is an upper bound on the set considered above, then a presentation of $x$ is eventually contained in either $F_{d}$ or $F_{m}$. It follows that the covering map of $(F_{d+m}, \mathcal{L}_{d+m})$ is injective, so it presents an SFT.
\end{proof}


\begin{example}
\label{RJ_ex_pos_det+nc_bf}
The next step is to use Lemma \ref{RJ_lem_addition_Xd} to construct renewal systems that share features with both $\RJXd{n_1, \dots, n_k}$ and the renewal systems considered in Example \ref{RJ_ex_pos_det}.
Given $n_1, \ldots , n_k \geq 2$ with $\max_j n_j > 2$, consider the list $L_d$ defined in (\ref{RJ_eq_Ld}) which generates the renewal system $\RJXd{n_1, \ldots, n_k}$,
and the list $L$ from (\ref{RJ_eq_pos_det}).
$L$ is left--modular, and $\mathsf{X}(L)$ is an SFT, so Lemma \ref{RJ_lem_addition_Xd} can be used to find the left Fischer cover of the SFT renewal system $X_+$ generated by
$
L_{+} = L_d \cup L \cup_{i=1}^k \{ a_i w \mid w \in L \}
$,
and the corresponding symbolic adjacency matrix is
\begin{displaymath}
A_+ =
\left( \!\!\! \begin{array}{c | c c c c | c c c c c | c | c c c c c  }
       b & \alpha    & 0 & b +\beta & a+\alpha                        &  b & 0 &\! \cdots \!& 0 & 0 & \!\cdots\! & b & 0  &\! \cdots \!& 0 & 0 
       \\
       \hline
       0            &   0      & \gamma &  0             & 0                 & 0 & 0           &\! \cdots \!& 0 & 0 &   & 0 & 0 &\! \cdots \!& 0 & 0      \\
       \beta      &   0      & 0            &  0             & \beta           & \beta & 0     &\! \cdots \!& 0 & 0 &  & \beta & 0 &\! \cdots \!& 0 & 0              \\
       0            &   0      & 0            &  0             & \tilde \alpha & 0 & 0           &\! \cdots \!& 0 & 0 &  & 0 & 0 &\! \cdots \!& 0 & 0      \\
       \gamma &   0      & 0            &  \gamma  & \gamma      & \gamma & 0 &\! \cdots \!& 0 & 0 &  & \gamma & 0 &\! \cdots \!& 0 & 0   \\
       \hline                          
       a_1     & 0 & 0 & a_1 & a_1 & 0     & 0  &  & 0 & 0 &  & a_1  & a_1 &\! \cdots \!& a_1 & a_1   \\       
       0         & 0 & 0 & 0     & 0     & a_1 & 0  &  & 0 & 0 &  & 0      & 0     &  & 0     & 0   \\ 
       \vdots & \vdots   & \vdots  & \vdots   & \vdots        &        &     &\! \ddots \!&    &    &  &         &        &\! \ddots \!&        &      \\          
       0         & 0 & 0 & 0     & 0     & 0     & 0  &  & 0 & 0 &  & 0      & 0     &  & 0     & 0   \\ 
       0         & 0 & 0 & 0     & 0     & 0     & 0  &  & a_1 & 0 &  & 0      & 0     &  & 0     & 0   \\
       \hline 
       \vdots &    &    &        &        &        &     &              &        &        & \!\ddots\!  &         &        &  &        &      \\                 
       \hline
       a_k     & 0 & 0 & a_k & a_k & a_k & a_k&\! \cdots \!& a_k & a_k &  & 0      & 0     & & 0     & 0   \\       
       0         & 0 & 0 & 0     & 0     & 0     & 0    &  & 0     & 0     &  & a_k      & 0     & & 0     & 0   \\ 
       \vdots & \vdots   & \vdots  & \vdots   & \vdots     &        &       &\! \ddots \!&        &        &  &         &        &\! \ddots \!&        &      \\          
       0         & 0 & 0 & 0     & 0     & 0     & 0    &  & 0     & 0     &  & 0      & 0     & & 0     & 0   \\ 
       0         & 0 & 0 & 0     & 0     & 0     & 0    &  & 0 & 0     &  & 0      & 0     & & a_k     & 0   \\
\end{array} \!\!\! \right) \; ,
\end{displaymath}
where $b = a + \alpha + \tilde \alpha$. Let $Y_+$ be a renewal system obtained from $X_+$ by a fragmentation of $a$, $\alpha$, $\tilde \alpha$, $\beta$, and $\gamma$. Then the (non--symbolic) adjacency matrix of the left Fischer cover of $Y_+$ is obtained from the matrix $A_+$ above by replacing $a_1, \ldots, a_k$ by $1$, and replacing $a$, $\alpha$, $\tilde \alpha$, $\beta$, and $\gamma$ by positive integers. Let $B_+$ be a matrix obtained in this manner. By doing row and column operations as in the construction that leads to the proof Proposition \ref{RJ_prop_non-cyclic_BF}, and by disregarding rows and columns where the only non--zero entry is a diagonal $1$, it follows that
\begin{displaymath}
\RJId - B_+ \rightsquigarrow 
\left( \!\!\! \begin{array}{c | c c c c | c c c c  }
       1-b            & -\alpha & 0            & -b -\beta   & -a-\alpha     &  -b          & -b             &\! \cdots \! & -b \\
       \hline
       0               & 1          & -\gamma &  0             & 0                 & 0            & 0              &\! \cdots \! & 0  \\  
       -\beta         & 0          & 1             &  0            & -\beta          & -\beta     & -\beta      &\! \cdots \! & -\beta  \\
       0              & 0          & 0              &  1             & -\tilde \alpha & 0            & 0            &\! \cdots \! & 0 \\
       -\gamma  & 0          & 0              & - \gamma & 1-\gamma     & -\gamma & -\gamma &\! \cdots \! & -\gamma \\
       \hline                          
       -1              & 0         & 0            & -1              & -1                 & 1            & 1-n_2     & \! \cdots \! & 1-n_k \\  
       -1              & 0         & 0            & -1              & -1                 & 1-n_1     & 1            &                  & 1-n_k \\
        \vdots      & \vdots & \vdots    & \vdots       & \vdots          &  \vdots    &               & \! \ddots \! & \vdots \\
        -1             & 0         & 0            & -1              & -1                 & 1-n_1     & 1-n_2     & \! \cdots \! & 1 \\           
\end{array} \!\!\! \right) \;  .
\end{displaymath}
Add the third row to the first and subtract the first column from columns $4, \ldots , k+4$ as in the proof of Lemma \ref{RJ_lem_entropy_cyclic} 
and choose the variables $a$, $\alpha$, $\tilde \alpha$, $\beta$, and $\gamma$ as in the proof of Theorem \ref{RJ_thm_rs_det_range}.  Assuming that $n_i | n_{i-1}$ for $2 \leq i \leq k$, this matrix can be  reduced to 
\begin{multline*}
\RJId - B_+ \rightsquigarrow \\
\left( \begin{array}{c | c c c c c }
x         & -1       & -1     & -1     & \cdots & -1 \\
\hline
2x-1    &  0       & -n_2 & -n_3 & \cdots & -n_k \\
0         & -n_1   &  n_2 &  0     &            & 0  \\
0         & -n_1   &  0     & n_3  &            & 0  \\
\vdots & \vdots &        &         & \ddots  & \vdots  \\
0        & -n_1     &  0    & 0      & \cdots  & n_k  \\ 
\end{array}  \right)
\rightsquigarrow 
\left( \begin{array}{c c c c c c }
x         & -\sum_{i=1}^k \frac{n_1}{n_i} & -1     & 0      & \cdots & 0 \\
2x-1    &  -(k-1)n_1                               & 0      &  0     & \cdots & 0 \\
0         &  0                                            &  n_2 &  0     &            & 0  \\
0         &  0                                            &  0     & n_3  &            & 0  \\
\vdots & \vdots                                     &         &         & \ddots  & \vdots   \\
0        &    0                                           &  0    & 0      &  \cdots  & n_k  \\ 
\end{array}  \right) \; ,
\end{multline*}
where $x \in \mathbb{Z}$ is arbitrary. Hence, the determinant is 
\begin{equation}
\label{RJ_eq_pos_det+nc_bf_det}
  \det(\RJId - B_+) =  n_2 \cdots n_k  \left( (2x-1)\sum_{i=1}^k \frac{n_1}{n_i}   -  x(k-1)n_1 \right) \; ,
  \end{equation}
and there exists an abelian group $G$ with at most two generators such that the Bowen--Franks group of the corresponding SFT is $G \oplus \mathbb{Z} / n_3 \mathbb{Z} \oplus \cdots \oplus \mathbb{Z} / n_k \mathbb{Z}$.
For $x = 0$, the determinant is negative and the Bowen--Franks group is $\mathbb{Z} \big/ \big(\sum_{i=1}^k \frac{n_2 n_1}{n_i} \big)\mathbb{Z} \oplus \mathbb{Z} /n_3\mathbb{Z} \oplus \cdots \oplus \mathbb{Z} /n_k\mathbb{Z}$.
\end{example}

This gives the first example of SFT renewal systems that simultaneously have positive determinants and non--cyclic Bowen--Franks groups.

\begin{theorem}\label{RJ_thm_nc_bf_det}
Given $n_1, \ldots, n_k \geq 2$ with $n_i | n_{i-1}$ for $2 \leq i \leq k$ there exist abelian groups $G_\pm$ with at most two generators and SFT renewal systems $\mathsf{X}(L_\pm)$ such that $\RJBF_+(\mathsf{X}(L_\pm)) = \pm G_\pm \oplus \mathbb{Z} /n_1\mathbb{Z} \oplus \cdots \oplus \mathbb{Z} /n_k \mathbb{Z}$.
\end{theorem}

\begin{proof}
Consider the renewal system from Example \ref{RJ_ex_pos_det+nc_bf}. Given the other variables, (\ref{RJ_eq_pos_det+nc_bf_det}) shows that $x$ can be chosen such that the determinant has either sign.
\end{proof}

The question raised by Adler, and the related question concerning the flow equivalence of renewal systems are still unanswered, and a significant amount of work remains before they can be solved.
However, there is hope that the techniques developed in Sec. \ref{RJ_sec_rs_lfc} and the special classes of renewal systems considered in Sec. \ref{RJ_sec_rs_range} can act as a foundation for the construction of a class of renewal systems attaining all the values of the Bowen--Franks invariant realised by irreducible SFTs.




\end{document}